\documentclass[dvips,a4paper,11pt]{amsart}

\usepackage[english]{babel} %[ngerman] for german

\usepackage{amssymb,amsmath}

\usepackage{bbm} %for characteristic function 1 with \mathbbm{1}
\usepackage{enumerate}
\usepackage{graphicx}
\usepackage{float}
\usepackage{enumitem} %ermoeglicht formation von itemize und enumerate, zB abstand davor kleiner machen: [noitemsep,topsep=0pt]
\usepackage{mathtools} %for \mathclap{} \mathrlap \mathllap for better spacing in combination with \underbrace
\usepackage{xcolor} %to highlight text like \tw
\usepackage{soul} %to highlight text/make gray background
\sethlcolor{lightgray} %sets the highlight color of the soul package
\usepackage{hyperref}
\usepackage{mathrsfs} %for \mathrsf font
\usepackage{cite} %for bibtex usage
\usepackage{verbatim}        % for "comment"

 \usepackage{multicol}
 \usepackage[bottom]{footmisc}

\usepackage{url}

\usepackage{amsthm}
\theoremstyle{plain}
\newtheorem{theorem}{Theorem}[section]
\newtheorem{lemma}[theorem]{Lemma}
\newtheorem{proposition}[theorem]{Proposition}
\newtheorem{corollary}[theorem]{Corollary}
\theoremstyle{remark}
\newtheorem{remark}[theorem]{Remark}
\theoremstyle{definition}

\newtheorem{example}[theorem]{Example}
\newcommand{\cS}{\mathcal{S}}

\newcommand{\Z}{\mathbb{Z}}
\newcommand{\R}{\mathbb{R}}
\newcommand{\C}{\mathbb{C}}

\newcommand{\tmu}{\tilde\mu}

\newcommand{\T}{\mathbb{T}}

\newcommand{\tw}{\textcolor{black}}

\makeatletter
\def\blfootnote{\gdef\@thefnmark{}\@footnotetext}
\makeatother

\DeclareMathOperator{\linspan}{{\mathrm{span}}}

\DeclareMathOperator{\diag}{{\mathrm{diag}}}
\DeclareMathOperator{\argmin}{{\mathrm{argmin}}}

\renewcommand{\d}[1][x]{\,\operatorname{d}\!#1}
\newcommand{\ddt}{\frac{\d[]}{\d[t]}}
\newcommand{\ddb}{\frac{\partial}{\partial b}}
\newcommand{\ip}[2]{\langle {#1},\, {#2} \rangle}
\newcommand{\norm}[1]{\| {#1} \|}
\DeclareMathOperator{\trace}{Tr}
\renewcommand{\AA}{\mathbf{A}}
\newcommand{\BB}{\mathbf{B}}
\newcommand{\CC}{\mathbf{C}}
\newcommand{\DD}{\mathbf{D}}
\newcommand{\EE}{\mathbf{E}}
\newcommand{\II}{\mathbf{I}}
\newcommand{\PP}{\mathbf{P}}

\newcommand{\VV}{\mathbf{V}}
\newcommand{\WW}{\mathbf{W}}

\numberwithin{equation}{section} %numbers equations with section# first

\begin{document}

\title[On optimal decay estimates for ODEs and PDEs]{On optimal decay estimates for ODEs and PDEs with modal decomposition}

\author{Franz Achleitner}
\address{Faculty of Mathematics, University of Vienna,\\ Oskar-Morgenstern-Platz~1, 1090 Vienna, Austria}
\email{franz.achleitner@univie.ac.at}

\author{Anton Arnold}
\address{TU Wien, Institute for Analysis and Scientific Computing, Wiedner Hauptstrasse 8-10, 1040 Vienna, Austria}
\email{anton.arnold@tuwien.ac.at}

\author{Beatrice Signorello}
\address{TU Wien, Institute for Analysis and Scientific Computing, Wiedner Hauptstrasse 8-10, 1040 Vienna, Austria}
\email{beatrice.signorello@tuwien.ac.at}

\begin{abstract}
We consider the Goldstein-Taylor model, which is a 2-velocity BGK model,
 and construct the ``optimal'' Lyapunov functional to quantify the convergence to the unique normalized steady state. 
The Lyapunov functional is optimal in the sense that
 it yields decay estimates in $L^2$-norm with the sharp exponential decay rate and minimal multiplicative constant.
The modal decomposition of the Goldstein-Taylor model leads to the study of a family of 2-dimensional ODE systems.
Therefore we discuss the characterization of ``optimal'' Lyapunov functionals for linear ODE systems with positive stable diagonalizable matrices.
We give a complete answer for \tw{optimal decay rates of} 2-dimensional ODE systems, 
 and a partial answer for higher dimensional ODE systems.
\end{abstract}


\thanks{All authors were supported by the FWF-funded SFB \#F65. 
The second author was partially supported by the FWF-doctoral school W1245 ``Dissipation and dispersion in nonlinear partial differential equations''.
We are grateful to the anonymous referee who led us to better distinguish the different cases studied in \S3 and \S4.}

\maketitle

%%%%%%%%%%%%%%%%%%%%%%%%%%%%%%%%%%%%%%%%%%%%%%%%%%%%%%%%%%%%%%%%%%%%%%%%%%%%%%%%%%%%%%%%%%%%%%%%

\section{Introduction}%\label{sec:intor}
%Introduction: We build on \cite{Achleitner2016, Dolbeault2015}: Want to construct Lyapunov functional for toy model in order to approach more complex BGK models with entropy method better (?) + sensitivity analysis.\marginpar{TODO: extend introduction} 

This note is concerned with optimal decay estimates of hypocoercive evolution equations that allow for a modal decomposition. 
The notion \emph{hypocoercivity} was introduced by Villani in \cite{ViH06} for equations of the form $\ddt{f}=-Lf$ on some Hilbert space $H$, 
where the generator $L$ is not coercive, but where solutions still exhibit exponential decay in time. 
More precisely, there should exist constants $\lambda>0$ and $c\ge1$, such that
\begin{equation}\label{exp-decay}
  \|e^{-Lt} f^I\|_{\widetilde H} \le c\,e^{-\lambda t} \|f^I\|_{\widetilde H}\qquad  \forall\,f^I\in \widetilde H\,,
\end{equation}
where $\widetilde H$ is a second Hilbert space, densely embedded in $(\ker L)^\perp\subset H$.

The large-time behavior of many hypocoercive equations have been studied in recent years, 
including Fokker-Planck equations \cite{ViH06, ArEr14, AAS}, kinetic equations \cite{DMS15} and BGK equations \cite{AAC16, AAC17}. 
Determining the sharp (i.e. maximal) exponential decay rate $\lambda$ was an issue in some of these works, in particular \cite{ArEr14, AAC16, AAC17}. 
But finding at the same time the smallest multiplicative constant $c\ge1$, is so far an open problem. 
And this is the topic of this note. 
For simple cases we shall describe a procedure to construct the ``optimal'' Lyapunov functional that will imply \eqref{exp-decay} with the sharp constants $\lambda$ and $c$.

For illustration purposes we shall focus here only on the following 2-velocity BGK-model (referring to the physicists Bhatnagar, Gross and Krook \cite{BGK}) for the two functions $f_\pm(x,t)\geq 0$ on the one-dimensional torus $x\in\T$ and for $t\ge0$. It reads
\begin{equation} \label{bgk}
  \begin{cases}
\partial_t f_+ &= - \partial_x f_+ + \frac12(f_--f_+)\,,\\
\partial_t f_- &= \partial_x f_- - \frac12(f_--f_+)\,.
\end{cases}
\end{equation}
This system of two transport-reaction equations is also called \emph{Goldstein-Taylor model}.
%probability density

For initial conditions normalized as $\int_0^{2\pi} \left[f^I_+(x)+f^I_-(x)\right]\d[x]=2\pi$, 
the solution $f(t)=(f_+(t),f_-(t))^\top$ converges to its unique (normalized) steady state with $f_+^\infty=f_-^\infty=\frac12$. 
\tw{The operator norm of the propagator for~\eqref{bgk} % $\|f(t)-f^\infty\|_{L^2(0,2\pi;\R^2)}$
 can be computed explicitly from the Fourier modes, see~\cite{MiMo2013}. By contrast, the goal of this paper and of \cite{AAC16, DMS15} is to refrain from explicit computations of the solution and to use Lyapunov functionals instead. Following this strategy, }
an explicit exponential decay rate of this two velocity model %by means of Lyapunov functionals 
was shown in \cite[\S 1.4]{DMS15}. 
The sharp exponential decay estimate was found in \cite[\S 4.1]{AAC16} via a refined functional, 
yielding the following result:
\begin{theorem}[{\cite[Th. 6]{AAC16}}]\label{th:bgk}
%\begin{theorem}[\cite[Theorem 6]{AAC15}]\label{th:bgk}
Let $f^I\in L^2(0,2\pi;\R^2)$. Then the solution to \eqref{bgk} satisfies
\begin{equation*} %\label{est:GT}
  \|f(t)-f^\infty\|_{L^2(0,2\pi;\R^2)} \le c\,e^{-\lambda t}\|f^I-f^\infty\|_{L^2(0,2\pi;\R^2)}\,, \qquad t\ge0\,,
\end{equation*}
with the optimal constants $\lambda=\frac12$ and $c=\sqrt3$.
\end{theorem}
\begin{remark}\label{rem1.2}
\begin{enumerate}
\item[a)] 
%Using the Goldstein-Taylor model as a simple example to improve the approach via Lyapunov functionals,
Actually, the optimal $c$ was not specified in \cite{AAC16},
 but will be the result of \tw{Theorem~\ref{thm:bestConst-2D}} below. 
\item[b)] As we shall illustrate in \S\ref{sec:ODE-ex},
 it does \emph{not} make sense to optimize these two constants at the same time. The optimality in Theorem \ref{th:bgk} refers to first maximizing the exponential rate $\lambda$, and then to minimize the multiplicative constant $c$.
\end{enumerate}
\end{remark}

The proof of Theorem \ref{th:bgk} is based on the spatial Fourier transform of \eqref{bgk}, cf. \cite{DMS15, AAC16}. We denote the Fourier modes in the discrete velocity basis $\{ \binom{1}{1},\,\binom{1}{-1}\}$ by $u_k(t)\in\C^2,\,k\in\Z$. They evolve according to the ODE systems
\begin{equation}\label{2vel-trans}
    \ddt{u_k}= -\CC_k\,u_k\,,\quad 
    \CC_k=\left(\begin{array}{cc}
    0 & ik \\
    ik & 1
\end{array}\right)\,,\quad k\in\Z\,,
\end{equation}
and their (normalized) steady states are 
$$
  u_0^\infty=\binom{1}{0}\,;\qquad u_k^\infty=\binom{0}{0},\quad k\ne0\,.
$$

In the main body of this note we shall construct appropriate Lyapunov functionals for such ODEs, 
in order to obtain sharp decay rates of the form \eqref{exp-decay}. 
In the context of the BGK-model \eqref{bgk},
combining such decay estimates for all modes $u_k$ then yields Theorem \ref{th:bgk}, \tw{as they are uniform in $k$.} 
We remark that the construction of Lyapunov functionals to reveal optimal decay rates in ODEs was already included in the classical textbook \cite[\S22.4]{Arnold1978}, but optimality of the multiplicative constant $c$ was not an issue there.

In this article we shall first review, from \cite{AAC16, AAC17}, 
the construction of Lyapunov functionals for linear first order ODE systems that reveal the sharp decay rate. 
\tw{They are quadratic functionals represented by some Hermitian matrix $\PP$.}
As these functionals are not uniquely determined, 
we shall then \tw{discuss a strategy to find the} ``best Lyapunov'' functional in \S\ref{sec:min}---by minimizing the \tw{condition number $\kappa(\PP)$. 
The method of \S\ref{sec:min} always yields an upper bound for the minimal multiplicative constant~$c$ 
 and the sharp constant in certain subcases (see Theorem~\ref{thm:bestConst-2D}).
The refined method of \S\ref{sec:min:2} covers another subclass (see Theorem~\ref{thm:bestConst-2D:2}). 
Overall we shall determine the optimal constant $c$} 
for 2-dimensional ODE systems,
and \tw{give estimates for it in} higher dimensions.
In the final section \S\ref{sec:ODE-ex} 
we shall illustrate how to obtain a whole family of decay estimates---with suboptimal decay rates, but improved constant $c$. 
For small time this improves the estimate obtained in~\S\ref{sec:min}.

%%%%%%%%%%%%%%%%%%%%%%%%%%%%%%%%%%%%%%%%%%%%%%%%%%%%%%%%%%%%%%%%%%%%%%%%%%%%%%%%%%%%%%%%%%%%%%%%

\section{Lyapunov Functionals for Hypocoercive ODEs}\label{sec:Lyap}

In this section we review decay estimates for linear ODEs with constant coefficients of the form
%$f(t)=(f_1(t),\,f_2(t),\ldots,f_n(t))^\top\in \C^n$:
 \begin{equation} \label{ODE:general:n}
  \begin{cases}
   \ddt{f}= -\CC f,\quad t\ge 0\,,\\
   f(0)   = f^I \in\C^n\,,
  \end{cases}
 \end{equation}
for some (typically non-Hermitian) matrix $\CC\in\C^{n\times n}$. 
\tw{To ensure that the origin is the unique asymptotically stable steady state,}
 we assume that the matrix~$\CC$ is \emph{hypocoercive} (i.e. positive stable, meaning that all eigenvalues have positive real part). 
%Then the origin is the unique and asymptotically stable steady state $f^\infty=0$ of~$\eqref{ODE:general:n}$.
Since we shall \emph{not} require that $\CC$ is coercive (meaning that its Hermitian part would be positive definite),  %$\C_H:=\frac12(\CC+\CC^*)>0$)
we \emph{cannot} expect that all solutions to \eqref{ODE:general:n} satisfy for the Euclidean norm: $\|f(t)\|_2\le e^{-\widetilde\lambda t}\|f^I\|_2$ for some $\widetilde\lambda>0$. However, such an exponential decay estimate does hold in an adapted norm that can be used as a Lyapunov functional. 

The construction of this Lyapunov functional is based on the following lemma:

\begin{lemma}[{\cite[Lemma 2]{AAC16}}, {\cite[Lemma 4.3]{ArEr14}}] \label{lemma:Pdefinition}
For any fixed matrix $\CC\in\C^{n\times n}$,
 let $\mu:=\min\{\Re(\lambda)|\lambda$ is an eigenvalue of $\CC\}$. 
Let $\{\lambda_{j}|1\leq j\leq j_0\}$ be all the eigenvalues of $\CC$ with $\Re(\lambda_j)=\mu$. 
%only counting their geometric multiplicity.
If all $\lambda_j$ ($j=1,\dots,j_0$) are non-defective\footnote{An eigenvalue is defective if its geometric multiplicity is strictly less than its algebraic multiplicity.},
  then there exists a positive definite Hermitian matrix $\PP\in\C^{n\times n}$ with
  \begin{align} \label{matrixestimate1}
   \CC^*\PP+\PP \CC &\geq 2\mu \PP\,,
  \end{align}
but $\PP$ is not uniquely determined.
%  where $\CC^*$ denotes the Hermitian transpose of $\CC$.
% \item[(ii)] If $\lambda_m$ is defective for at least one $m\in\{1,\dots,m_0\}$, then for any $\eps>0$ there exists a positive definite Hermitian matrix $\PP=\PP(\eps)\in\C^{n\times n}$ with
% \begin{align}
% \label{degeneratematrixestimate} \CC^*\PP+\PP \CC &\geq 2(\mu-\eps) \PP\,.
% \end{align}
%  \end{itemize}

 Moreover, if all eigenvalues of $\CC$ are non-defective, 
 examples of such matrices $\PP$ satisfying~\eqref{matrixestimate1} are given by
 \begin{align} \label{simpleP1} 
  \PP:= \sum\limits_{j=1}^n b_j \, w_j \otimes w_j^* \,,
 \end{align}
 where $w_j\in\C^{n}$ ($j=1,\dots,n$) denote the (right) normalized eigenvectors of $\CC^*$ (i.e. $\CC^* w_j =\bar\lambda_j w_j$),
 and $b_j\in\R^+$ ($j=1,\dots,n$) are arbitrary weights.
\end{lemma}

For $n=2$ all positive definite Hermitian matrices $\PP$ satisfying~\eqref{matrixestimate1} have the form~\eqref{simpleP1},
but for $n\geq 3$ this is not true 
(see Lemma~\ref{bestP-2D-diagonalB} and \tw{Example~\ref{ex:matrixP:3D}}, respectively).

In this article, for simplicity, 
we shall only consider the case when all eigenvalues of $\CC$ are non-defective. 
For the extension of Lemma \ref{lemma:Pdefinition} and of the corresponding decay estimates to the defective case we refer to \cite[Prop. 2.2]{AAS} and \cite{AJW18}.

\medskip
%Now we consider examples, where all eigenvalues of $\CC\in\C^{n\times n}$ are non-defective and have positive real parts. 
Due to the positive stability of $\CC$, the origin is the unique and asymptotically stable steady state $f^\infty=0$ of~$\eqref{ODE:general:n}$:
Due to Lemma~\ref{lemma:Pdefinition}, 
there exists a positive definite Hermitian matrix $\PP\in\C^{n\times n}$ such that $\CC^* \PP +\PP\CC \geq 2\mu \PP$ where $\mu = \min \Re(\lambda_j) >0$.
Thus, the time derivative of the adapted norm $\|f\|^2_\PP := \ip{f}{\PP f}$ along solutions of \eqref{ODE:general:n} satisfies
 \begin{align*}
  \ddt \|f(t)\|^2_\PP 
%     &= \ddt \ip{f(t)}{\PP f(t)} = \ip{\ddt f(t)}{\PP f(t)} + \ip{f(t)}{\PP \ddt f(t)} \nonumber \\
%     &= -\ip{\CC f(t)}{\PP f(t)} -\ip{f(t)}{\PP \CC f(t)} = -\ip{f(t)}{(\CC^* \PP +\PP \CC)f(t)} \,. \label{GDF:derivative}
    &\leq -2\mu \|f(t)\|^2_\PP\,. % \qquad \text{with} \quad \mu = \min \Re(\lambda_j),
 \end{align*} 
%which implies 
\tw{Hence the evolution becomes a contraction in the adapted norm:}
\begin{equation} \label{ODE-normP-decay}
 \|f(t)\|^2_\PP\leq e^{-2\mu t} \|f^I \|^2_\PP \,,\qquad t\ge0\,.
\end{equation}  
\tw{Clearly, this procedure can yield the sharp decay rate $\mu$, only if $\PP$ satisfies \eqref{matrixestimate1}.}

Next we translate this decay in $\PP$-norm into a decay in the Euclidean norm:
\begin{equation} \label{ODE:exponentialDecay}
 \|f(t)\|^2_2\leq (\lambda^\PP_{\min})^{-1} \|f(t)\|^2_\PP \leq  (\lambda^\PP_{\min})^{-1} e^{-2\mu t} \|f^I \|^2_\PP \leq  \kappa(\PP)\, e^{-2\mu t} \|f^I \|^2_2 \,,\quad t\ge0\,,
\end{equation}  
where $0<\lambda^\PP_{\min}\le\lambda^\PP_{\max}$ are, respectively, the smallest and largest eigenvalues of $\PP$, and $\kappa(\PP)={\lambda^\PP_{\max}}/{\lambda^\PP_{\min}}$ is the (numerical) condition number of $\PP$ with respect to the Euclidean norm.
\tw{While~\eqref{ODE-normP-decay} is sharp, \eqref{ODE:exponentialDecay} is not necessarily sharp:}
Given the spectrum of \tw{$\CC$}, the exponential decay rate in \eqref{ODE:exponentialDecay} is optimal,
\tw{but the multiplicative constant not necessarily.
For the optimality of the chain of inequalities in \eqref{ODE:exponentialDecay}
 we have to distinguish two scenarios: 
Does there exist an initial datum $f^I$ such that each inequality will be (simultaneously) an equality for some \emph{finite} $t_0\ge0$~? 
Or is this only possible asymptotically as $t\to\infty$~? 
We shall start the discussion with the former case, which is simpler, and defer the latter case to \S\ref{sec:min:2}. 
The first scenario allows to find the optimal multiplicative constant for $\CC\in\R^{2\times2}$, based on \eqref{ODE:exponentialDecay}. 
But in other cases it may only yield an explicit upper bound for it, as we shall discuss in \S\ref{sec:min:2}.}

\tw{Concerning the first inequality of \eqref{ODE:exponentialDecay}, a solution $f(t_0)$ will satisfy $\|f(t_0)\|^2_2= (\lambda^\PP_{\min})^{-1} \|f(t_0)\|^2_\PP$ for some $t_0\geq 0$
 only if $f(t_0)$ is in the eigenspace associated to the eigenvalue $\lambda^\PP_{\min}$ of~$\PP$.
Moreover, the initial datum $f^I$ satisfies $\|f^I \|^2_\PP = \lambda^\PP_{\max} \|f^I \|^2_2$
 if $f^I$ is in the eigenspace associated to the eigenvalue $\lambda^\PP_{\max}$ of $\PP$.
Finally we consider the second inequality of \eqref{ODE:exponentialDecay}:
If the matrix $\CC$ satisfies, e.g., $\Re\lambda_j=\mu>0;\,j=1,...,n$, with all eigenvalues non-defective, then we always have 
\begin{equation} \label{est:middle}
  \|f(t)\|^2_\PP = e^{-2\mu t} \|f^I \|^2_\PP \qquad \forall t\geq 0\,,
\end{equation}
since \eqref{matrixestimate1} is an equality then. This is the case for our main example \eqref{2vel-trans} with $k\ne0$.}

\tw{Since the matrix $\PP$ is not unique, we shall now discuss the choice of $\PP$ as }
to minimize the multiplicative constant in \eqref{ODE:exponentialDecay}. 
To this end we need to find the matrix $\PP$ with minimal condition number that satisfies \eqref{matrixestimate1}. 
Clearly, the answer can only be unique up to 
%the multiplication by a positive constant,
a positive multiplicative constant,
since $\widetilde\PP:=\tau\PP$ with $\tau>0$ would reproduce the estimate \eqref{ODE:exponentialDecay}.

As we shall prove in \S\ref{sec:min}, the answer to this minimization problem is very easy in 2 dimensions: 
The best $\PP$ corresponds to equal weights in \eqref{simpleP1}, e.g. choosing $b_1=b_2=1$.

%%%%%%%%%%%%%%%%%%%%%%%%%%%%%%%%%%%%%%%%%%%%%%%%%%%%%%%%%%%%%%%%%%%%%%%%%%%%%%%%%%%%%%%%%%%%%%%%

\section{\tw{Optimal Constant via} Minimization of the Condition Number}\label{sec:min}

In this section, 
we describe a procedure \tw{towards constructing ``optimal'' Lyapunov functionals: 
For solutions $f(t)$ of ODE~\eqref{ODE:general:n} they will imply %\eqref{exp-decay}
 \begin{equation} \label{exp-decay:c}
   \|f(t)\|_2\le c\,e^{-\mu t}\|f^I\|_2
 \end{equation}
with the sharp constant $\mu$ and partly also the sharp constant $c$.}

We shall describe the procedure for ODEs~\eqref{ODE:general:n} with positive stable matrices~$\CC$.
For simplicity we confine ourselves to diagonalizable matrices~$\CC$ (i.e. all eigenvalues are non-defective). 
In this case, Lemma~\ref{lemma:Pdefinition} states that \tw{there exist positive definite Hermitian matrices $\PP$ 
 satisfying the matrix inequality~\eqref{matrixestimate1}.
Following \eqref{ODE:exponentialDecay}, $\sqrt{\kappa(\PP)}$ is always an upper bound for the constant $c$ in~\eqref{exp-decay:c}.}
 %Then, \eqref{ODE:exponentialDecay} shows
% that the constant $c$ in~\eqref{exp-decay:c} is given by $c=\sqrt{\kappa(\PP)}$.
% where $\kappa(\PP)$ is the condition number of the matrix $\PP$ with respect to the Euclidean norm.
% Since  $\PP$ is a positive definite Hermitian matrix, 
%  $\kappa(\PP)={\lambda^\PP_{max}}/{\lambda^\PP_{min}}$
%  where $\lambda^\PP_{min}$, $\lambda^\PP_{max}$ are, respectively, the smallest (positive) and largest eigenvalues of $\PP$.
Our strategy is now to minimize $\kappa(\PP)$ on the set of \tw{all admissible matrices~$\PP$}. %---for any $n\geq 2$.
We shall prove that this actually yields the minimal constant $c$ %for $n=2$ 
 \tw{in certain cases (see Theorem~\ref{thm:bestConst-2D}).}
% We shall describe the procedure for ODEs~\eqref{ODE:general:n} with positive stable matrices~$\CC$.
% For simplicity we confine ourselves to diagonalizable matrices~$\CC$ (i.e. all eigenvalues are non-defective). 
% In this case, Lemma~\ref{lemma:Pdefinition} states that all matrices $\PP$ of form~\eqref{simpleP1}
%  satisfy the matrix inequality~\eqref{matrixestimate1}.
% \tw{Following \eqref{ODE:exponentialDecay}, $c:=\sqrt{\kappa(\PP)}$ is always an upper bound for the constant $c$ in~\eqref{exp-decay:c}.}
%  %Then, \eqref{ODE:exponentialDecay} shows
% % that the constant $c$ in~\eqref{exp-decay:c} is given by $c=\sqrt{\kappa(\PP)}$.
% % where $\kappa(\PP)$ is the condition number of the matrix $\PP$ with respect to the Euclidean norm.
% % Since  $\PP$ is a positive definite Hermitian matrix, 
% %  $\kappa(\PP)={\lambda^\PP_{max}}/{\lambda^\PP_{min}}$
% %  where $\lambda^\PP_{min}$, $\lambda^\PP_{max}$ are, respectively, the smallest (positive) and largest eigenvalues of $\PP$.
% Our strategy is now to minimize $\kappa(\PP)$ on the set of matrices of form~\eqref{simpleP1}---for any $n\geq 2$.
% We shall prove that this actually yields the minimal constant $c$ for $n=2$ that is achievable via the chain of inequalities \eqref{ODE:exponentialDecay}.
% But for $n\geq 3$ 
%  this procedure does not always yield the minimal $\kappa$ for all matrices satisfying~\eqref{matrixestimate1}.
In 2 dimensions this minimization problem can be solved very easily
 \tw{thanks to Lemma~\ref{bestP-2D-diagonalB} and Lemma~\ref{bestP-2D}: }
\begin{lemma}\label{bestP-2D-diagonalB}
 Let $\CC\in\C^{2\times 2}$ be a diagonalizable, positive stable matrix.
 Then all matrices~$\PP$ satisfying~\eqref{matrixestimate1} are of the form~\eqref{simpleP1}.
\end{lemma}
\begin{proof}
We use again the matrix $\WW$ whose columns are the normalized (right) eigenvectors of $\CC^*$
 such that
 \begin{equation} \label{C*:JordanForm}
  \CC^* \WW = \WW \DD^* \,,
 \end{equation}
 with $\DD=\diag(\lambda_1^\CC,\lambda_2^\CC)$
 where $\lambda_j^\CC$ ($j\in\{1,2\}$) are the eigenvalues of $\CC$.
Since $\WW$ is regular, $\PP$ can be written as
 \[ \PP = \WW \BB \WW^* \,, \]
 with some positive definite Hermitian matrix~$\BB$.
Then the matrix inequality~\eqref{matrixestimate1} can be written as 
%  \[ 0 \leq (\CC^*-\mu\II)\PP+\PP (\CC-\mu\II) 
%       = (\CC^*-\mu\II) \WW \BB \WW^* + \WW \BB \WW^* (\CC-\mu\II) 
%  \]
 \begin{align*} % \label{matrixestimate2a}
  2\mu \WW \BB \WW^*
   &\leq \CC^*\WW \BB \WW^* +\WW \BB \WW^* \CC %\\
%   &= \WW \DD^* \BB \WW^* +\WW \BB \DD \WW^*
    = \WW (\DD^* \BB +\BB \DD) \WW^* \,.
 \end{align*}
This matrix inequality is equivalent to
 \begin{equation} \label{matrixestimate2}
  0 \leq (\DD^*-\mu\II) \BB + \BB (\DD -\mu\II) \,.
 \end{equation}
Next we order the eigenvalues $\lambda_j^\CC$ ($j\in\{1,2\}$) of $\CC$
 increasingly with respect to their real parts, 
 such that $\Re(\lambda_1^\CC)=\mu$.
Moreover, we consider 
\begin{equation*} %\label{matrixB}
 \BB =\begin{pmatrix} b_1 & \beta \\ \overline{\beta} & b_2 \end{pmatrix} 
\end{equation*}
 where $b_1,b_2>0$ and $\beta\in\C$ with $|\beta|^2 < b_1 b_2$.
Then the right hand side of~\eqref{matrixestimate2} is 
%  \[
%    (\DD^*-\mu\II) \BB + \BB (\DD -\mu\II) 
%      = \begin{pmatrix}
%          0 & [(\Re\lambda_2^\CC-\mu) +\ii (\Im\lambda_2^\CC -\Im\lambda_1^\CC)] \beta \\
%          [(\Re\lambda_2^\CC-\mu) -\ii (\Im\lambda_2^\CC -\Im\lambda_1^\CC)] \overline{\beta} & 2 (\Re\lambda_2^\CC-\mu) b_2 
%         \end{pmatrix} 
%  \]
 \begin{equation} \label{matrixestimate2b}
   (\DD^*-\mu\II) \BB + \BB (\DD -\mu\II) 
     = \begin{pmatrix}
         0 & (\lambda_2^\CC-\lambda_1^\CC) \beta \\
         \overline{(\lambda_2^\CC-\lambda_1^\CC) \beta}\ & 2 b_2 \Re(\lambda_2^\CC-\lambda_1^\CC)
        \end{pmatrix} 
 \end{equation}
 with $\trace [(\DD^*-\mu\II) \BB + \BB (\DD -\mu\II)] = 2 b_2 \Re(\lambda_2^\CC-\lambda_1^\CC)$ 
 and 
 \[ \det [(\DD^*-\mu\II) \BB + \BB (\DD -\mu\II)]
 %     = - \big|((\Re\lambda_2^\CC-\mu) +\ii (\Im\lambda_2^\CC -\Im\lambda_1^\CC)) \beta \big|^2
     = - \big| \lambda_2^\CC-\lambda_1^\CC \big|^2 |\beta|^2 \,.
 \]
Condition~\eqref{matrixestimate2} is satisfied if and only if 
 $\trace [(\DD^*-\mu\II) \BB + \BB (\DD -\mu\II)]\geq 0$ which holds due to our assumptions on $\lambda_2^\CC$ and $b_2$,
 and $\det [(\DD^*-\mu\II) \BB + \BB (\DD -\mu\II)]\geq 0$.
The last condition holds if and only if 
 \[ \lambda_2^\CC =\lambda_1^\CC \quad \text{or } \beta =0\,. \]
% However in the former case $\lambda_1^\CC=\lambda_2^\CC$, hence, 
% $\PP=\II$ is an admissible matrix with optimal condition number $1$.
% Thus the statement (which is equivalent to $\beta =0$) follows. 
In the latter case $\BB$ is diagonal and hence $\PP$ is of the form~\eqref{simpleP1}.
In the former case, % $\lambda_1^\CC =\lambda_2^\CC$, hence~
 \eqref{C*:JordanForm} shows that $\CC =\lambda_1^\CC \II$,
 and the inequality~\eqref{matrixestimate1} is trivial.
Now any positive definite Hermitian matrix $\PP$ has a diagonalization $\PP =\VV \EE \VV^*$,
 with a diagonal real matrix $\EE$ and an orthogonal matrix $\VV$,
 whose columns are --of course-- eigenvectors of $\CC$.
Thus, $\PP$ is again of the form~\eqref{simpleP1}.
\qed
\end{proof}
\tw{In contrast to this 2D result,}
 in dimensions $n\geq 3$ there exist matrices~$\PP$ satisfying~\eqref{matrixestimate1} which are not of form~\eqref{simpleP1}:
\begin{example} \label{ex:matrixP:3D}
Consider the matrix $\CC =\diag(1,2,3)$.
Then, all matrices 
 \begin{equation} \label{matrix:P:block-diagonal} 
  \PP(b_1,b_2,b_3,\beta) = \begin{pmatrix} b_1 & 0 & 0 \\ 0 & b_2 & \beta \\ 0 & \beta & b_3 \end{pmatrix}
 \end{equation}
 with positive $b_j$ ($j\in\{1,2,3\}$) and $\beta\in\R$ such that $8b_2 b_3 -9\beta^2\geq 0$,
 are positive definite Hermitian matrices
 and satisfy~\eqref{matrixestimate1} for $\CC =\diag(1,2,3)$ and $\mu=1$.
But the eigenvectors of $\CC^*$ are the canonical unit vectors.
Hence, matrices of form~\eqref{simpleP1} would all be diagonal. 
\qed
\end{example}
 
\tw{Restricting the minimization problem to admissible matrices~$\PP$ of form~\eqref{simpleP1} we find:}
Defining a matrix $\WW:=(w_1|\ldots|w_n)$ whose columns are the (right) normalized eigenvectors of $\CC^*$
 allows to rewrite formula~\eqref{simpleP1} as
 \begin{equation} \label{def:P:matrixW:2}
 \begin{split}
   \PP &= \sum\limits_{j=1}^n b_j \, w_j \otimes w_j^*
        = \WW \diag(b_1,b_2,\ldots,b_n) \WW^* \\
       &= \big(\WW \diag(\sqrt{b_1},\sqrt{b_2},\ldots,\sqrt{b_n})\big) \big(\WW \diag(\sqrt{b_1},\sqrt{b_2},\ldots,\sqrt{b_n})\big)^*
 \end{split}
 \end{equation}
with positive constants $b_j$ ($j=1,\ldots,n$).
The identity
 \[ \WW \diag(\sqrt{b_1},\sqrt{b_2},\ldots,\sqrt{b_n}) = (\sqrt{b_1} w_1 | \ldots | \sqrt{b_n} w_n ) \]
 shows that the weights are just rescalings of the eigenvectors.
Finally, the condition number of $\PP$ is the squared condition number of $(\WW \diag(\sqrt{b_1},\sqrt{b_2},\ldots,\sqrt{b_n}))$.
Hence, to find matrices $\PP$ of form~\eqref{def:P:matrixW:2} with minimal condition number,
 is equivalent to identifying (right) precondition matrices
 among the positive definite diagonal matrices
% of the form $\diag(\sqrt{b_1},\sqrt{b_2},\ldots,\sqrt{b_n})$
 which minimize the condition number of $\WW$.
This minimization problem can be formulated as a convex optimization problem~\cite{BrMo94}
 based on the result~\cite{SeOv1990}.
Due to \cite[Theorem 1]{Bu68},
 the minimum is attained (i.e.\ an optimal scaling matrix exists)
 since our matrix $\WW$ is non-singular. 
(Note that its column vectors form a basis of $\C^n$.)
The convex optimization problem can be solved by standard software
 providing also the exact scaling matrix which minimizes the condition number of $\PP$,
 see the discussion and references in~\cite{BrMo94}.
For more information on convex optimization and numerical solvers, see e.g.~\cite{BoVa04}.

\tw{We return to the minimization of $\kappa(\PP)$ in 2 dimensions:}
\begin{lemma}\label{bestP-2D}
 Let $\CC\in\C^{2\times 2}$ be a diagonalizable, positive stable matrix.
 Then the condition number of the associated matrix~$\PP$ in~\eqref{simpleP1} is minimal
 by choosing equal weights, e.g. $b_1=b_2=1$.
\end{lemma}
\begin{proof}
A diagonalizable matrix $\CC$ has only non-defective eigenvalues.
Up to a unitary transformation,
 we can assume w.l.o.g. that the eigenvectors of $\CC^*$ are 
%Rotating the coordinate system such that the first eigenvector $w_1$ becomes the first basis vector allows to assume w.l.o.g. that
 \begin{equation} \label{eigenvectors:newBasis}
 w_1 = \begin{pmatrix} 1 \\ 0 \end{pmatrix}\,, \quad w_2 = \begin{pmatrix} \alpha \\ \sqrt{1-\alpha^2} \end{pmatrix} \quad \text{for some $\alpha\in[0,1)$.} 
 \end{equation}
This unitary transformation describes the change of the coordinate system.
To construct the new basis,
 we choose one of the normalized eigenvectors~$w_1$ as first basis vector,
 and recall that the second normalized eigenvector~$w_2$ is only determined up to a scalar factor $\gamma\in\C$ with $|\gamma|=1$. 
The right choice for the scalar factor $\gamma$ allows to fulfill the above restriction on $\alpha$.
 
We use the representation of the positive definite matrix~$\PP$ in~\eqref{def:P:matrixW:2}:
\begin{equation} \label{matrixW}
  \PP = \WW \diag(b_1,b_2) \WW^* \quad \text{with }
    \WW = \begin{pmatrix} 1 & \alpha \\ 0 & \sqrt{1-\alpha^2} \end{pmatrix} \,.
\end{equation} 
Since $\PP$ and $\tau\PP$ have the same condition number,
 we consider w.l.o.g. $b_1=1/b$ and $b_2=b$.
% since~$\PP$ is only determined up to a positive multiplicative constant.
%% We can choose $b_1=1/b$ and $b_2=b$
%%  since multiplying matrix~$\PP$ with a positive scalar yields just another admissible matrix $\widetilde\PP$.
Thus, we have to determine the positive parameter $b>0$ which minimizes the condition number of 
 \begin{equation} \label{def:matrixP:b}
  \PP(b) = \WW \diag(1/b,b) \WW^*
         = \begin{pmatrix}
            \tfrac1b + b\alpha^2 & b\alpha \sqrt{1-\alpha^2} \\
            b\alpha \sqrt{1-\alpha^2} & b (1-\alpha^2)
           \end{pmatrix}\,.
 \end{equation}
The condition number of matrix $\PP(b)$ is given by
 \begin{equation*} %\label{kappa:P} 
  \kappa(\PP(b)) =\lambda^\PP_+(b) /\lambda^\PP_-(b) \ge1 \,,
 \end{equation*}
 where
 \[ \lambda_\pm^\PP (b) = \frac{\trace\PP(b)\pm \sqrt{(\trace\PP(b))^2 -4\det \PP(b)}}{2} \]
 are the (positive) eigenvalues of $\PP(b)$.  
% The matrix $\PP$ is a positive definite Hermitian matrix.
% Therefore its condition number $\kappa(\PP)$ with respect to the Euclidean norm is determined as 
%  \begin{equation} \label{kappa:P} 
%   \kappa(\PP) =\lambda^\PP_+ /\lambda^\PP_- \ge1
%  \end{equation}
%  where $0 < \lambda^\PP_- \leq \lambda^\PP_+$ are the positive eigenvalues of $\PP$ in increasing order.
% The characteristic polynomial associated to $\PP(b)$ is given by 
%  \[ \lambda^2 -\trace\PP(b) \lambda +\det \PP(b) = 0 \]
%  since $\trace\PP(b) = b+1/b$ and $\det\PP(b)=1-\alpha^2$.
% Thus, the positive eigenvalues of $\PP(b)$ are  
%  \[ \lambda_\pm = \frac{\trace\PP(b)\pm \sqrt{(\trace\PP(b))^2 -4\det \PP(b)}}{2}\,. \]
We notice that $\trace\PP(b) = b+1/b$ is independent of $\alpha$
 and is a convex function of $b\in(0,\infty)$ 
 which attains its minimum for $b=1$.
Moreover, $\det\PP(b)=1-\alpha^2$ is independent of $b$.
This implies that the condition number 
 \[
  \kappa(\PP(b)) =\frac{\lambda^\PP_+(b)}{\lambda^\PP_-(b)}
              =\frac{1 + \sqrt{1 -\frac{4\det \PP(b)}{(\trace\PP(b))^2}}}{1 - \sqrt{1 -\frac{4\det \PP(b)}{(\trace\PP(b))^2}}}
 \]
 attains its \tw{unique} minimum at $b=1$, taking the value 
 \begin{equation} \label{kappa:min}
  \kappa_{\min}=\frac{1+\alpha}{1-\alpha}\,.
 \end{equation}
 \qed
\end{proof}

\tw{
This 2D-result does not generalize to higher dimensions.
In dimensions $n\geq 3$ there exist diagonalizable positive stable matrices~$\CC$,
 such that the matrix $\PP$ with equal weights~$b_j$ does not yield the lowest condition number
 among all matrices of form~\eqref{simpleP1}.}
We give a counterexample in 3 dimensions: 
\begin{example} \label{ex:matrixP:3D:2}
For some $\CC^*$, consider its eigenvector matrix 
 \begin{equation} \label{matrix:W}
   \WW := \begin{pmatrix} 1 & 1 & 1 \\ 0 & 1 & 1 \\ 0 & 0 & 1 \end{pmatrix}
   \diag\Bigg(1,\frac{1}{\sqrt{2}},\frac{1}{\sqrt{3}} \Bigg) \,,
 \end{equation}
 which has normalized column vectors.
We define the matrices $\PP(b_1,b_2,b_3) :=\WW \diag(b_1,b_2,b_3) \WW^*$
 for positive parameters $b_1$, $b_2$ and $b_3$,
 which are of form~\eqref{simpleP1}
 and hence satisfy the inequality~\eqref{matrixestimate1}.
In case of equal weights $b_1=b_2=b_3$
 the condition number is $\kappa(\PP(b_1,b_1,b_1)) \approx 15.12825876$.
But using %$\kappa(\WW \WW^*) =\kappa(\WW^* \WW)$ and 
 \cite[Theorem 3.3]{Ko2006},
 the minimal condition number $\min_{b_j} \kappa(\PP(b_1,b_2,b_3)) \approx 13.92820324$ 
 is attained for the weights $b_1=2$, $b_2=4$ and $b_3=3$.
\qed
\end{example}

\tw{Combining Lemma~\ref{bestP-2D-diagonalB} and Lemma~\ref{bestP-2D} we have
\begin{corollary}
Let $\CC\in\C^{2\times 2}$ be a diagonalizable, positive stable matrix.
Then the condition number is minimal among all matrices~$\PP$ satisfying~\eqref{matrixestimate1},
 if~$\PP$ is of form~\eqref{simpleP1} with equal weights, e.g. $b_1=b_2=1$. 
\end{corollary}}
\tw{
This 2D-result does not generalize to higher dimensions.
Extending the conclusion of Example~\ref{ex:matrixP:3D:2},
 we shall now show that $\PP$ does not necessarily have to be of form~\eqref{simpleP1},
 if its condition number should be minimal:}
\begin{example}
\tw{We consider a special case of Example~\ref{ex:matrixP:3D:2}, with}
 \begin{equation*}
  \widetilde{\CC} = (\WW^*)^{-1} \diag(1,2,3) \WW^*
 \end{equation*}
 with $\WW$, the eigenvector matrix of $\widetilde{\CC}^*$, given by~\eqref{matrix:W}.
Then the matrices $\widetilde{\CC}$ and
 \begin{equation*} %\label{matrix:P:beta}
  \widetilde\PP(b_1,b_2,b_3,\beta) := \WW \PP(b_1,b_2,b_3,\beta) \WW^*
 \end{equation*}
 with matrix~$\PP(b_1,b_2,b_3,\beta)$ in~\eqref{matrix:P:block-diagonal} satisfy the matrix inequality~\eqref{matrixestimate1} with~$\mu=1$.
But $\widetilde\PP$ is not of form~\eqref{simpleP1} if $\beta\ne 0$.
Nevertheless, the condition number $\kappa(\widetilde\PP(b_1,b_2,b_3,\beta)) \approx 5.82842780720132$ 
 for the weights $b_1=2$, $b_2=4$, $b_3=3$, and $\beta=-2.45$,
 is much lower than with $\beta=0$
 (i.e. $\kappa(\widetilde\PP(2,4,3,0)) \approx 13.92820324$, cf. Example~\ref{ex:matrixP:3D:2}).
\qed
\end{example}
 
\tw{Lemma~\ref{bestP-2D} and inequality~\eqref{ODE:exponentialDecay} show that
 $\sqrt{\kappa_{\min}}$ from~\eqref{kappa:min} is an \emph{upper bound} for the best constant in~\eqref{exp-decay:c} for the 2D case.
For matrices with eigenvalues that have the same real part
 it actually yields the minimal multiplicative constant~$c$,
 as we shall show now.
Other cases will be discussed in \S\ref{sec:min:2}.}%end\tw

\tw{For a diagonalizable matrix $\CC\in\C^{2\times 2}$ with $\lambda_1^\CC=\lambda_2^\CC$ it holds that 
$\|f(t)\|_2 = e^{-\Re \lambda_1^\CC t}  \|f^I\|_2$. 
And for the general case we have:
}%end\tw
\tw{
\begin{theorem}\label{thm:bestConst-2D}
% Let $\CC\in\C^{2\times 2}$ be a diagonalizable, positive stable matrix with $\Re\lambda_1^\CC=\Re\lambda_2^\CC$.
Let $\CC\in\C^{2\times 2}$ be a diagonalizable, positive stable matrix with eigenvalues $\lambda_1^\CC \ne\lambda_2^\CC$,
  and associated eigenvectors $v_1$ and $v_2$, resp.
If the eigenvalues have identical real parts, i.e. $\Re\lambda_1^\CC=\Re\lambda_2^\CC$,
 then the condition number of the associated matrix~$\PP$ in~\eqref{simpleP1} with equal weights, e.g. $b_1=b_2=1$,
 yields the minimal constant %as $c=\sqrt{\kappa(\PP)}$
 in the decay estimate~\eqref{exp-decay:c} for the ODE \eqref{ODE:general:n}:
 \begin{equation} \label{min:c:IR+DI} % Distinct Real-parts, Identical Imaginary-parts
   c =\sqrt{\kappa(\PP)} =\sqrt{\frac{1+\alpha}{1-\alpha}} \quad \text{where } \alpha := \Big|\Big\langle {\frac{v_1}{\|v_1\|}}\,,{\frac{v_2}{\|v_2\|}} \Big\rangle\Big| \,.
 \end{equation}
\end{theorem}
\begin{proof}
With the notation from the proof of Lemma \ref{bestP-2D} we have
 \begin{equation*} %\label{def:P:b}
  \PP(1) = \begin{pmatrix}
            1 + \alpha^2 & \alpha \sqrt{1-\alpha^2} \\
            \alpha \sqrt{1-\alpha^2} & 1-\alpha^2
           \end{pmatrix}\,,
 \end{equation*}
with the eigenvectors $y_+^\PP=(\sqrt{1-\alpha^2},1-\alpha)^\top$, $y_-^\PP=(\sqrt{1-\alpha^2},-1-\alpha)^\top$. According to the discussion after \eqref{ODE:exponentialDecay} we choose the initial condition $f^I=y_+^\PP$. From the diagonalization \eqref{C*:JordanForm} of $\CC$ we get
$$
  f(t)=(\WW^*)^{-1} e^{-\DD t}\WW^* f^I\,.
$$
Using \eqref{matrixW} and $\WW^* y_\pm^\PP=\sqrt{1-\alpha^2} \binom{1}{\pm1}$ we obtain directly that
$$
  f(t_0)=e^{-\lambda_1^\CC\,t_0} y_-^\PP\quad \mbox{with }\:t_0=\frac{\pi}{|\Im(\lambda_2^\CC-\lambda_1^\CC)|}\,.
$$
Hence, also the first inequality in \eqref{ODE:exponentialDecay} is sharp at $t_0$. Sharpness of the whole chain of inequalities then follows from \eqref{est:middle}, and this finishes the proof.
 \qed
\end{proof}
}

This theorem now allows us to identify the minimal constant $c$ in Theorem~\ref{th:bgk} \tw{on the Goldstein-Taylor model}:
\tw{The eigenvalues of the matrices $\CC_k,\,k\ne0$ from \eqref{2vel-trans} are $\lambda=\frac12\pm i\sqrt{k^2-\frac14}$.} 
The corresponding transformation matrices $\PP_k$ with $b_1=b_2=1$ are given by $\PP_0=\II$ and 
\begin{equation*}%\label{Pk}
    \PP_k=\begin{pmatrix}
    1 & -\frac{i}{2k} \\
    \frac{i}{2k} & 1
    \end{pmatrix}\,,
    %\PP_k=\left(\begin{array}{cc}
    %1 & -\frac{i}{2k} \\
    %\frac{i}{2k} & 1
%\end{array}\right)\,,
\quad\mbox{ with }\quad \kappa(\PP_k)=\frac{2|k|+1}{2|k|-1}\,,\qquad k\ne0\,.
\end{equation*}
Combining the decay estimates for all Fourier modes $u_k(t)$ shows 
that the minimal multiplicative constant in Theorem \ref{th:bgk} is given by $c=\sqrt{\kappa(\PP_{\pm1})}=\sqrt3$. 
For a more detailed presentation 
 how to recombine the modal estimates
 we refer to \S4.1 in \cite{AAC16}.
\medskip

\section{Optimal Constant for 2D Systems}\label{sec:min:2}
\tw{
The optimal constant $c$ in~\eqref{exp-decay:c} for $\CC\in\C^{2\times 2}$ with $\Re\lambda_1^\CC=\Re\lambda_2^\CC$ was determined in Theorem~\ref{thm:bestConst-2D}.
In this section we shall discuss the remaining 2D cases.
We start to derive the minimal multiplicative constant~$c$ for matrices~$\CC$ with eigenvalues
 that have distinct real parts but identical imaginary parts.
\begin{theorem}\label{thm:bestConst-2D:2}
 Let $\CC\in\C^{2\times 2}$ be a diagonalizable, positive stable matrix with eigenvalues $\lambda_1^\CC$ and $\lambda_2^\CC$,
  and associated eigenvectors $v_1$ and $v_2$, resp.
 If the eigenvalues have distinct real parts $\Re\lambda_1^\CC <\Re\lambda_2^\CC$ and identical imaginary parts $\Im\lambda_1^\CC=\Im\lambda_2^\CC$,
  then the minimal multiplicative constant~$c$ in~\eqref{exp-decay:c} for the ODE~\eqref{ODE:general:n}
  is given by
  \begin{equation} \label{min:c:DR+II} % Distinct Real-parts, Identical Imaginary-parts
   c = \frac1{\sqrt{1-\alpha^2}} \quad \text{where } \alpha := \Big|\Big\langle {\frac{v_1}{\|v_1\|}}\,,{\frac{v_2}{\|v_2\|}} \Big\rangle\Big| \,.
  \end{equation}
\end{theorem}
}
\begin{proof}
\tw{We use again the unitary transformation as in the proof of Lemma~\ref{bestP-2D},
 such that the eigenvectors $w_1$ and $w_2$ of $\CC^*$ are given in~\eqref{eigenvectors:newBasis}.
If $f(t)$ is a solution of~\eqref{ODE:general:n}, then $\tilde{f}(t) := e^{i \Im\lambda_1^\CC t} f(t)$ satisfies
 \begin{equation} \label{ODE:realEV}
  \ddt \tilde{f}(t) = -\widetilde{\CC} \tilde{f}(t)\,, \quad 
  \tilde{f}(0) = f^I \,,
 \end{equation}
 with
 \[
  \widetilde{\CC} := (\CC -i \Im\lambda_1^\CC \II) = (\WW^*)^{-1} \begin{pmatrix} \Re\lambda_1^\CC & 0 \\ 0 & \Re\lambda_2^\CC \end{pmatrix} \WW^* \,.
 \]
The multiplication with $e^{i \Im\lambda_1^\CC t}$ is another unitary transformation 
 and does not change the norm, i.e. $\|f(t)\|_2 = \|\tilde{f}(t)\|_2$.
Therefore, we can assume w.l.o.g. that matrix~$\CC$ has real coefficients \underline{and} distinct real eigenvalues.
Then, % for initial data $f^I\in \C^2$ with $f^I = \Re f^I + i \Im f^I$
 the solution $f(t)$ of the ODE~\eqref{ODE:general:n}
 satisfies $\Re f(t) = f_{re}(t)$ and $\Im f(t) = f_{im}(t)$
 where $f_{re}(t)$ and $f_{im}(t)$ are the solutions of the ODE~\eqref{ODE:general:n} with initial data $\Re f^I$ and $\Im f^I$, resp.
Altogether, we can assume w.l.o.g. that \underline{all} quantities are real valued:
\newline
Considering a matrix $\CC\in\R^{2\times 2}$ with two distinct real eigenvalues $\lambda_1<\lambda_2$ and real eigenvectors $v_1$ and $v_2$,
 then the associated eigenspaces $\linspan\{v_1\}$ and $\linspan\{v_2\}$ dissect the plane into four sectors
 \begin{equation} \label{def:sectors}
  \cS^{\pm \mp} := \{ z_1 v_1 +z_2 v_2 \ |\ z_1\in\R^\pm\,, \ z_2\in\R^\mp \} \,,
 \end{equation}
see Fig.~\ref{fig:SectorGamma}.
A solution~$f(t)$ of ODE~\eqref{ODE:general:n} starting in an eigenspace will approach the origin in a straight line, such that 
 \begin{equation} \label{decay}
 \|f(t)\|^2_2 = e^{-2\lambda_j^\CC t} \|f^I \|^2_2 \qquad \forall t\geq 0\,. 
 \end{equation}
If a solution starts instead in one of the four (open) sectors $\cS^{\pm \mp}$, 
 it will remain in that sector while approaching the origin.
In fact, since $\lambda_1^\CC<\lambda_2^\CC$,
 if $f^I = z_1 (v_1 +\gamma v_2)$ for some $z_1 \in\R\setminus\{0\}$ and $\gamma \in\R$,
 then the solution %$f(t)$
 \[ f(t) = z_1 \big( e^{-\lambda_1^\CC t} v_1 + \gamma e^{-\lambda_2^\CC t} v_2 \big)
         = z_1 e^{-\lambda_1^\CC t} \big( v_1 + \gamma e^{-(\lambda_2^\CC -\lambda_1^\CC) t} v_2 \big)
 \]
 of the ODE~\eqref{ODE:general:n} will remain in the sector
 \begin{equation} \label{def:sectors:gamma}
  \cS^{\pm}_\gamma := \{ z_1( v_1 +z_2 v_2) \ |\ z_1\in\R^\pm\,, \ z_2\in [\min(0,\gamma),\max(0,\gamma)] \} \,,
 \end{equation}
 see Fig.~\ref{fig:SectorGamma}.
\begin{figure}[!th]
\includegraphics[width=\textwidth]{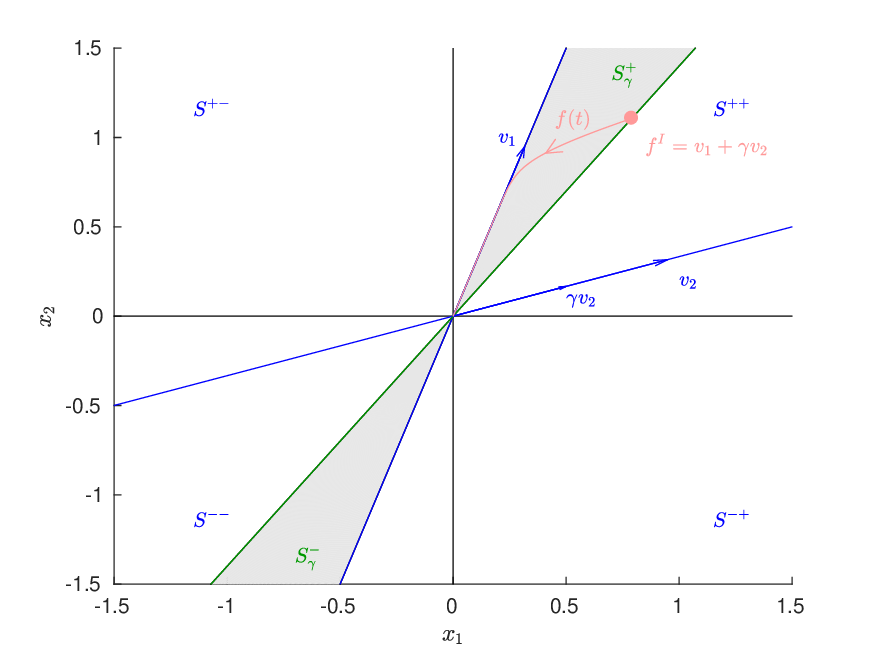}
\caption{\tw{
 The blue (black) lines are the eigenspaces $\linspan\{v_1\}$ and $\linspan\{v_2\}$ of matrix~$\CC$.
 The red (grey) curve is a solution $f(t)$ of the ODE~\eqref{ODE:general:n} with initial datum $f^I$.
 The \emph{shaded regions} are the sectors~$S^+_\gamma$, $S^-_\gamma$ with the choice $\gamma=1/2$.
 Note: The curves are colored only in the electronic version of this article.}
}
\label{fig:SectorGamma}
\end{figure} 
}
\tw{For a fixed $f^I =z_1 (v_1 +\gamma v_2)$, let $\cS$ be the corresponding sector $\cS^{\pm}_\gamma$. 
Then estimate~\eqref{ODE:exponentialDecay} can be improved as follows
\begin{equation} \label{ODE:exponentialDecay:2}
 \|f(t)\|^2_2\leq \frac{1}{\lambda^\PP_{\min,\, \cS}} \|f(t)\|^2_\PP \leq \frac{e^{-2\mu t}}{\lambda^\PP_{\min,\, \cS}}  \|f^I \|^2_\PP \leq  c_\cS (\PP)\, e^{-2\mu t} \|f^I \|^2_2 \,,\quad t\ge0\,,
\end{equation}  
where 
\begin{equation} \label{def:kappa:S}
   \lambda^\PP_{\min,\, \cS} := \inf_{x\in\cS} \frac{\ip{x}{\PP x}}{\ip{x}{x}}\,, \quad 
   \lambda^\PP_{init,\, \cS} := \frac{\ip{f^I}{\PP f^I}}{\ip{f^I}{f^I}}\,, \quad 
   c_\cS (\PP) := \frac{\lambda^\PP_{init,\, \cS}}{\lambda^\PP_{\min,\, \cS}} \,.
\end{equation}
Note that, in the definition of $\lambda^\PP_{init,\, \cS}$
 the sector $\cS\in\left\{\cS_\gamma^\pm\big| \gamma\in\R \right\}$ also determines corresponding initial conditions  $f^I\in\partial \cS$ via $f^I =z_1 (v_1 +\gamma v_2)$
 (up to the constant $z_1\ne0$ which drops out in $\lambda^\PP_{init,\, \cS}$).}

\tw{For \eqref{ODE:exponentialDecay:2} to hold for all trajectories and one fixed constant on the right hand side,
 we have to take the supremum over all initial conditions or, equivalently,
 over all sectors $\cS\in\left\{\cS_\gamma^\pm\big| \gamma\in\R \right\}$. 
Although $f^I=z_2v_2$ is not included in any sector $\cS_\gamma^+$,
 its corresponding multiplicative constant 1 (see \eqref{decay}) is still covered.
Then, the minimal multiplicative constant in~\eqref{exp-decay:c} 
%  \begin{equation} \label{est:full} 
%   \|f(t)\|^2_2 \leq c\, e^{-2\mu t} \|f^I \|^2_2\,, \qquad \text{for all } t\geq 0 \,, 
%  \end{equation} 
 using~\eqref{ODE:exponentialDecay:2} is 
 \begin{equation} \label{best:constant:S}
  \widetilde{c} = \sqrt{\inf_{\PP}\, \sup_{\cS}\, c_\cS (\PP)} \,,
 \end{equation}
where $\PP$ ranges over all matrices of the form \eqref{simpleP1}.}\\

\noindent\tw{
\underline{Step 1} (computation of $\lambda_{\min,\, \cS^+_\gamma}^{\PP}$ for $\gamma$ fixed):
To find an explicit expression for this minimal constant $c$,
 we first determine $c_\cS (\PP)$ for a given admissible matrix~$\PP$.
As an example of sectors, we consider only $\cS^+_\gamma$ for fixed $\gamma\le0$ and compute 
\begin{align*}
 \lambda_{\min,\, \cS^+_\gamma}^{\PP} 
  &= \inf_{x\in \cS^+_\gamma} \frac{\ip{x}{\PP x}}{\norm{x}^2} 
   = \inf_{z_1\in\R^+,\, z_2\in[\gamma,0]} \frac{\ip{z_1 (v_1 +z_2 v_2)}{\PP (z_1 (v_1 +z_2 v_2))}}{\norm{z_1 (v_1 +z_2 v_2)}^2} \\
  &= \inf_{z_2\in[\gamma,0]} \frac{\ip{v_1 +z_2 v_2}{\PP (v_1 +z_2 v_2)}}{\norm{v_1 +z_2 v_2}^2} \,.
\end{align*}
This also shows that $\lambda_{\min,\, \cS^+_\gamma}^{\PP} =\lambda_{\min,\, \cS^-_\gamma}^{\PP}$
 for any fixed $\gamma\in\R$.
Next, we use the result of Lemma~\ref{bestP-2D-diagonalB} and~\eqref{def:P:matrixW:2},
 stating that the only admissible matrices are $\PP = \WW \diag(b_1,b_2) \WW^*$ for $b_1,b_2>0$.
Since $c_{\cS} (b\PP) = c_{\cS} (\PP)$ for all $b>0$,
 we consider w.l.o.g. $b_1=1/b$ and $b_2=b$ for $b>0$.
Then, we deduce
\begin{align*}
 \lambda_{\min,\, \cS^+_\gamma}^{\PP} 
  &= \inf_{z\in[\gamma,0]} \frac{\ip{v_1 +z v_2}{\PP (v_1 +z v_2)}}{\norm{v_1 +z v_2}^2} \\
  &= \inf_{z\in[\gamma,0]} \frac{\ip{\WW^* (v_1 +z v_2)}{\diag(1/b,b) \WW^* (v_1 +z v_2)}}{\norm{v_1 +z v_2}^2} \,.
\end{align*}
In our case of a real matrix~$\CC$ with distinct real eigenvalues, 
 the left and right eigenvectors are related as follows:
Up to a change of orientation, $\ip{w_j}{v_k} =\delta_{jk}$ ($j,k\in\{1,2\}$).
% \[ \ip{w_j}{v_k} = \begin{cases} 1 &\text{if } j=k\,, \\ 0 &\text{if } j\ne k\,. \end{cases} \]
Considering $\ip{w_j}{v_j}=1$ for $j=1,2$, implies that the vectors $w_j$ and $v_j$ can be normalized simultaneously only if matrix~$\CC$ is symmetric.
Therefore, using a coordinate system such that the normalized eigenvectors of $\CC^*$ are given as~\eqref{eigenvectors:newBasis} and $\VV:=(v_1|v_2) =(\WW^*)^{-1}$ yields
 \[ 
 v_1 = \frac1{\sqrt{1-\alpha^2}} \begin{pmatrix} \sqrt{1-\alpha^2} \\ -\alpha \end{pmatrix}\,, \quad
 v_2 = \frac1{\sqrt{1-\alpha^2}} \begin{pmatrix} 0 \\ 1 \end{pmatrix} \quad \text{for $\alpha$ in~\eqref{eigenvectors:newBasis}.} 
 \]
Finally, we obtain
\[
 \lambda_{\min,\, \cS^+_\gamma}^{\PP} 
   = \inf_{z\in[\gamma,0]} \frac{\ip{\WW^* (v_1 +z v_2)}{\diag(1/b,b) \WW^* (v_1 +z v_2)}}{\norm{v_1 +z v_2}^2} 
   = \inf_{z\in[\gamma,0]} g(z)
%   &= \sup_{z\in[\gamma,0]} \frac{(1-\alpha^2)\, \ip{\binom{1}{a}}{\diag(1/b,b) \binom{1}{a}}}{\norm{v_1 +z v_2}^2} \\
%   &= \sup_{z\in[\gamma,0]} \frac{(1-\alpha^2)\, (\frac{1}{b} + b a^2)}{1 +2a \ip{v_1}{v_2} + a^2} \\
%   &= \sup_{z\in[\gamma,0]} \frac{(1-\alpha^2)\, (\frac{1}{b} + b a^2)}{1 -2\alpha a + a^2}
\]
and $\lambda^\PP_{init,\, \cS^+_\gamma} = g(\gamma)$ with 
\begin{equation} \label{def:g} 
 g(z) := \frac{(1-\alpha^2)\, (\frac{1}{b} + b z^2)}{1 -2\alpha z + z^2} \,.
\end{equation}
% First, one notices that
% \begin{equation} \label{g:values}
%  g(0) = \frac{1-\alpha^2}{b}\,, \quad \lim_{c\to\pm\infty} g(c) = (1-\alpha^2) b \,.
% \end{equation}
% To find the extremal points in the interior, we differentiate $g(c)$ and obtain
% \[
%  g'(c) = \frac{(1-\alpha^2)}{(1 -2\alpha c + c^2)^2} \Big(-2\alpha b c^2 + 2\big(b-\frac1b\big) c +\frac2b \alpha\Big) \,.
% \]
% We observe $g'(0) = (1-\alpha^2) 2\alpha/b > 0$.
% Moreover, $g'(c)=0$ if and only if $c$ is a root of the quadratic polynomial $\big(-2\alpha b c^2 + 2\big(b-\frac1b\big) c +\frac2b \alpha\big)=0$.
% The roots are 
}

\noindent\tw{
\underline{Step 2} (extrema of the function $g$):
The function $g$ has local extrema at
\[
 z_\pm = \frac1{2\alpha b} \Big( b-\frac1b \pm \sqrt{ \big(b-\frac1b\big)^2 + 4\alpha^2} \Big)
\]
 which satisfy $z_- <0 <z_+$.
Writing $g'(z)=h_1(z)/h_2(z)$ with $h_1(z) :=\big(-2\alpha b z^2 + 2\big(b-\frac1b\big) z +\frac2b \alpha\big)$ and $h_2(z) := (1 -2\alpha z + z^2)^2/(1-\alpha^2) >0$,
 we derive
 \[ g''(z_\pm) = \frac{h_1'(z_\pm)}{h_2(z_\pm)} = \mp 2 \frac1{h_2(z_\pm)} \sqrt{ \big(b-\frac1b\big)^2 + 4\alpha^2} \,. \]
In fact, the function $g$ attains its global minimum on $\R$ (and on $\R_0^-$) at $z_-$, and its global maximum on $\R$ at $z_+$.
%To compute $c_{\cS^\pm_\gamma}(\PP(b))$ for $\gamma\in\R^-$,
%  we distinguish the cases $b\in (0,1)$, $b=1$, and $b\in(1,\infty)$.
% In case $b\in (0,1)$, we find that $g(0)=(1-\alpha^2)/b > \lim_{z\to-\infty} g(z) = (1-\alpha^2)b$.
% Due to our analysis ($g'(0)>0$ and $g''(z_-)>0$),
%  $g(z)$ attains its global minimum on $\R^-$ at $z_-$
%   and its global maximum at $0$.
% In case $b=1$, we derive $g(0) = \lim_{z\to-\infty} g(z) = (1-\alpha^2)$. 
% In case $b\in (1,\infty)$, we find that $g(0)=(1-\alpha^2)/b < \lim_{z\to-\infty} g(z) = (1-\alpha^2)b$.
% Due to our analysis ($g'(0)>0$ and $g''(z_-)>0$),
%  $g(z)$ attains its global minimum on $\R^-$ at $z_-$
%   but reaches its global supremum only in the limit $z\to-\infty$.
% we note that $g(z)$ still attains its global minimum on $\R^-$ at $z_-$.
The global supremum of $g(z)$ on $\R^-$ exists and satisfies
  \[ \sup_{z\in\R^-} g(z)
      = \begin{cases}
         g(0) =(1-\alpha^2)/b &\text{if } b\in(0,1) \,, \\
         g(0) = \lim_{z\to-\infty} g(z) =1-\alpha^2 &\text{if } b=1 \,, \\
         \lim_{z\to-\infty} g(z) =(1-\alpha^2)b &\text{if } b\in(1,\infty) \,.         
        \end{cases}
  \]
}

\noindent\tw{
\underline{Step 3} (optimization of $c_{\cS^\pm_\gamma}(\PP)$ w.r.t.\ $\gamma$):
We obtain
\[
 c_{\cS^\pm_\gamma}(\PP(b))
  = \frac{g(\gamma)}{\lambda_{\min,\, \cS^+_\gamma}^{\PP(b)}}
  = \begin{cases} 
     1 &\text{if } z_- \leq\gamma <0 \,, \\
     g(\gamma)/g(z_-) &\text{if } \gamma\leq z_- \,.
    \end{cases}
\]
Finally, we derive 
\begin{equation} \label{sup:cS:Rminus}
 \sup_{\gamma\in\R^-} c_{\cS^\pm_\gamma}(\PP(b))
  = \lim_{\gamma\to-\infty} \frac{g(\gamma)}{g(z_-)}
  = \frac{(1-\alpha^2)b}{g(z_-)} \,,
\end{equation}
and in a similar way,
\begin{equation} \label{sup:cS:Rplus}
 \sup_{\gamma\in\R^+} c_{\cS^\pm_\gamma}(\PP(b))
  = \frac{g(z_+)}{g(0)}
  = \frac{b g(z_+)}{1-\alpha^2} \,. 
\end{equation}
To finish this analysis we note that $c_{\cS^\pm_0}(\PP(b))=1$, due to \eqref{decay} and $f^I=z_1v_1$.
}

\noindent\tw{
\underline{Step 4} (minimization of $\sup_{\cS}\, c_\cS (\PP)$ w.r.t.\ $\PP$):
We obtain
\[
 \inf_{\PP}\, \sup_{\cS}\, c_\cS (\PP)
  = \inf_{b\in(0,\infty)} \sup_{\gamma\in\R} c_{\cS^\pm_\gamma}(\PP(b)) \\
  = \inf_{b\in(0,\infty)}
      \max\Big\{\frac{(1-\alpha^2)b}{g(z_-)},\, 1,\, \frac{b g(z_+)}{1-\alpha^2}\Big\} \,.
\]
Taking into account the $b$-dependence of $z_\pm$,
 the functions $\frac{(1-\alpha^2)b}{g(z_-)}$ and $\frac{b g(z_+)}{1-\alpha^2}$ are monotone increasing in~$b$,
 since 
 \[
   \ddb \frac{(1-\alpha^2)b}{g(z_-)} >0 \,, \qquad
   \ddb \frac{b g(z_+)}{1-\alpha^2}  >0 \,.
 \]
Therefore we have to study their limits as $b\to 0$:
We derive
\begin{equation} \label{limits:b}
\begin{split} 
 &\lim_{b\to 0} \frac{(1-\alpha^2)b}{g(z_-)} = 1 \qquad \qquad \text{using} \quad \lim_{b\to 0} z_-(b) = -\infty \,, \\
 &\lim_{b\to 0} \frac{b g(z_+)}{1-\alpha^2}  = \frac1{1-\alpha^2} >1 \qquad \text{using} \quad \lim_{b\to 0} z_+(b) = \alpha \,.
\end{split}
\end{equation}
Hence, $\inf_{b\in(0,\infty)}\, \sup_{\gamma\in\R}\, c_{\cS^\pm_\gamma} (\PP(b))$ is realized by 
 the sector $\cS^\pm_\gamma$ with $\gamma=z_+(b)>0$ and in the limit $b\to0$.
Altogether we obtain
 \[
  \widetilde{c} = \sqrt{\inf_{\PP}\, \sup_{\cS}\, c_\cS (\PP)} = \frac1{\sqrt{1-\alpha^2}} \,,
 \]
 where the first equality holds since we discussed all solutions.
This finishes the proof.
}

\noindent\tw{
\underline{Step 5:}
Finally we have to verify that $\widetilde{c}$ is minimal in~\eqref{exp-decay:c}.
We shall show that it is attained asymptotically (as $t\to\infty$) for a concrete trajectory:
For fixed $b\in(0,\infty)$, the minimal multiplicative constant in~\eqref{ODE:exponentialDecay:2} is attained for the solution with initial datum $f^I = v_1 +z_+(b) v_2=y_+^{\PP(b)}$,
 which is the eigenvector pertaining to the largest eigenvalue of $\PP(b)$ (cp. to the proof of Theorem~\ref{thm:bestConst-2D}). 
The formula for $f^I$ holds
 since $\sup_{\cS}\, c_\cS (\PP(b)) = bg(z_+(b))/(1-\alpha^2)$.
This can be verified by a direct comparison of~\eqref{sup:cS:Rminus} and~\eqref{sup:cS:Rplus}.
For $b$ small it also follows from~\eqref{limits:b}.
 In the limit $b\to0$, $\PP(b)$ in~\eqref{def:matrixP:b} approaches a multiple of $w_1 \otimes w_1^*$ and  
  \[ f^I = v_1 +z_+(b) v_2 \longrightarrow v_1 + \alpha v_2 = w_1 \,. \]
 The solution $f(t)$ of the ODE~\eqref{ODE:general:n} with $f^I = w_1$ satisfies 
  \begin{equation} \label{sol:f:w_1}
     f(t) = e^{-\CC t} w_1
      = \VV \begin{pmatrix} e^{-\lambda_1 t} & 0 \\ 0 & e^{-\lambda_2 t} \end{pmatrix} \WW^* w_1
      = e^{-\lambda_1 t} v_1 + \alpha e^{-\lambda_2 t} v_2 \,.
  \end{equation}
 This implies 
  \[ e^{\Re \lambda_1 t} \frac{\|f(t)\|_2}{\|f^I\|_2} 
      \leq \|v_1 +\alpha e^{-\Re(\lambda_2 -\lambda_1) t} v_2 \|_2
      \stackrel{t\to\infty}{\longrightarrow} \|v_1\|_2 =\tfrac1{\sqrt{1-\alpha^2}} 
  \]
  and it finishes the proof.
%   \[ \|f(t)\|_2 
%       \leq e^{-\Re \lambda_1 t} \|v_1 +\alpha e^{-\Re(\lambda_2 -\lambda_1) t} v_2 \|_2 
%       \leq \tfrac1{\sqrt{1-\alpha^2}}\ e^{-\Re \lambda_1 t} \|f^I\|_2 \,.
%   \]
\qed}
\end{proof}

\tw{After the analysis in Theorems~\ref{thm:bestConst-2D} and \ref{thm:bestConst-2D:2},
 we are left with the case of a matrix~$\CC\in\C^{2\times 2}$ with eigenvalues $\lambda_1$ and $\lambda_2$
 such that the real \underline{and} imaginary parts are distinct.
This case can not occur for real matrices~$\CC$.
The proof of Lemma~\ref{bestP-2D} gives an upper bound~$\sqrt{\tfrac{1+\alpha}{1-\alpha}}$ for the multiplicative constant in~\eqref{exp-decay:c}.
On the other hand, the solution $f(t)$ of the ODE~\eqref{ODE:general:n} with $f^I = w_1$ satisfies~\eqref{sol:f:w_1},
 hence, 
 \begin{align*}
  \|f(t)\|_2^2 
   &= e^{-2\Re \lambda_1 t} \|v_1 +\alpha e^{-(\lambda_2 -\lambda_1) t} v_2 \|_2^2 \\
   &= \tfrac1{1-\alpha^2} e^{-2\Re \lambda_1 t} \Big( 1 -2\alpha^2 e^{-\Re(\lambda_2-\lambda_1)t} \cos\big(\Im(\lambda_2-\lambda_1)t\big) +\alpha^2 e^{-2\Re(\lambda_2-\lambda_1)t} \Big) \,.
 \end{align*}
 The expression in the bracket is bigger than 1, e.g. at time $t=\pi/\Im(\lambda_2-\lambda_1)$.
 Thus the minimal multiplicative constant $c$ is definitely bigger than $\tfrac1{\sqrt{1-\alpha^2}}$, which is the best constant for $\Im \lambda_1=\Im \lambda_2$ (see Theorem~\ref{thm:bestConst-2D:2}). 
}

\medskip
\tw{
Next, we derive the upper and lower envelopes for the norm of solutions~$f(t)$ of ODE~\eqref{ODE:general:n} in order to determine the sharp constant $c$. 
For a diagonalizable matrix $\CC\in\C^{2\times 2}$ with $\lambda_1^\CC=\lambda_2^\CC$ it holds that 
$\|f(t)\|_2 = e^{-\Re \lambda_1^\CC t}  \|f^I\|_2$. 
And for the general case we have:
%minimal multiplicative constant~$c$ for matrices~$\CC$ with eigenvalues
% that have distinct real parts and distinct imaginary parts.
\begin{proposition}\label{bestConst-2D:3}
 Let $\CC\in\C^{2\times 2}$ be a diagonalizable, positive stable matrix with eigenvalues $\lambda_1^\CC\ne\lambda_2^\CC$,
  and associated eigenvectors $v_1$ and $v_2$, resp.
 Then the norm of solutions $f(t)$ of ODE~\eqref{ODE:general:n} satisfies 
  \[ h_-(t) \|f^I\|^2_2 \leq \|f(t)\|^2_2 \leq h_+(t) \|f^I\|^2_2 \,,\qquad \forall t\geq 0 \,, \]
 where the envelopes $h_\pm(t)$ are given by 
  \begin{equation*} %\label{envelope}
   h_\pm(t) := e^{-2\Re \lambda_1^\CC t} m_\pm(t)    
  \end{equation*}
%   \begin{equation*} %\label{envelope}
%    \begin{split}
%    & h_\pm(t) \\
%    &= \pm e^{-\Re (\lambda_1^\CC +\lambda_2^\CC)t} % \times \\
%    %&\quad
%    \Big( \sqrt{ \frac{\big(\cosh(\gamma t)-\alpha^2 \cos(\delta t)\big)^2}{(1-\alpha^2)^2} -1} \pm \frac{\big(\cosh(\gamma t)-\alpha^2 \cos(\delta t)\big)}{1-\alpha^2} \Big) \\ 
%    &= e^{-2\Re \lambda_1^\CC t} m_\pm(t)    
% %    &= \pm \frac{e^{-2\Re \lambda_1^\CC t} }{1-\alpha^2} \frac12 
% %          \Big( \sqrt{ \big(1 -2\alpha^2 \cos(\delta t) e^{-\gamma t} + e^{-2\gamma t}\big)^2 -4(1-\alpha^2)^2 e^{-2\gamma t}} \\
% %         & \qquad \qquad \qquad \quad \pm\big(1 -2\alpha^2 \cos(\delta t) e^{-\gamma t} + e^{-2\gamma t}\big) \Big)
%    \end{split}
%   \end{equation*}
  with
  \[ m_\pm(t)
      :=\pm e^{-\gamma t} \Big( \sqrt{ \frac{\big(\cosh(\gamma t)-\alpha^2 \cos(\delta t)\big)^2}{(1-\alpha^2)^2} -1} \pm \frac{\big(\cosh(\gamma t)-\alpha^2 \cos(\delta t)\big)}{1-\alpha^2} \Big) \,,
  \]
  where $\gamma:=\Re(\lambda_2^\CC -\lambda_1^\CC)$, $\delta:=\Im(\lambda_2^\CC -\lambda_1^\CC)$, $\alpha := \Big|\Big\langle {\frac{v_1}{\|v_1\|}}\,,{\frac{v_2}{\|v_2\|}} \Big\rangle\Big|$
  and $\alpha\in[0,1)$.
% If the eigenvalues have distinct real parts $\Re\lambda_1^\CC <\Re\lambda_2^\CC$ and distinct imaginary parts $\Im\lambda_1^\CC \ne\Im\lambda_2^\CC$,
%  then
\end{proposition}
While the rest of the article is based on estimating Lyapunov functionals,
 the following proof will use the explicit solution formula of the ODE.
}
\begin{proof}
\tw{We use again the unitary transformation as in the proof of Lemma~\ref{bestP-2D},
 such that the eigenvectors $w_1$ and $w_2$ of $\CC^*$ are given in~\eqref{eigenvectors:newBasis}.
If $f(t)$ is a solution of~\eqref{ODE:general:n}, then $\tilde{f}(t) = e^{\lambda_1^\CC t} f(t)$ satisfies
 \begin{equation} \label{ODE:reduced}
  \ddt \tilde{f}(t) = -\widetilde{\CC} \tilde{f}(t)\,, \quad 
  \tilde{f}(0) = f^I \,,
 \end{equation}
 with
 \[
  \widetilde{\CC} = (\CC -\lambda_1^\CC \II) = (\WW^*)^{-1} \begin{pmatrix} 0 & 0 \\ 0 & \lambda_2^\CC -\lambda_1^\CC \end{pmatrix} \WW^* \,.
 \]
The explicit solution~$\tilde{f}(t)$ of~\eqref{ODE:reduced} is 
 \[
  \tilde{f}(t)
   = (\WW^*)^{-1} \begin{pmatrix} 1 & 0 \\ 0 & e^{-(\gamma +i \delta) t} \end{pmatrix} \WW^* f^I
   = \begin{pmatrix} f^I_1 \\ \tfrac{\alpha}{\sqrt{1-\alpha^2}} (e^{-(\gamma +i \delta) t} -1) f^I_1 +e^{-(\gamma +i \delta) t} f^I_2 \end{pmatrix} \,,
 \]
 where $\gamma =\Re(\lambda_2^\CC -\lambda_1^\CC)$ and $\delta =\Im(\lambda_2^\CC -\lambda_1^\CC)$.
If the initial data $f^I$ lies in $\R\times\C$
 then the solution will satisfy $\tilde{f}(t) \in\R\times\C$ for all $t\geq 0$.
The multiplication with $\overline{f^I_1}/|f^I_1|$ is another unitary transformation and does not change the norm.
Therefore, to compute the envelope for the norm of solutions $\tilde{f}(t)$ of ODE~\eqref{ODE:reduced}
 we assume w.l.o.g. that 
 \begin{equation} \label{def:fI}
  f^I_{\phi,\theta} = \begin{pmatrix} \cos(\phi) \\ \sin(\phi) e^{i \theta} \end{pmatrix} \in \R\times\C \,,
  \qquad \text{where } \phi,\theta\in [0,2\pi)\,,
 \end{equation} 
 such that $\|f^I_{\phi,\theta}\|=1$.
We consider the solution $\tilde{f}_{\phi,\theta}(t)$ for~\eqref{ODE:reduced} with $f^I=f^I_{\phi,\theta}$.
To compute the envelopes (for fixed $t$), we solve 
 $\partial_\phi \|\tilde{f}_{\phi,\theta}\|^2 =0$ and 
 $\partial_\theta \|\tilde{f}_{\phi,\theta}\|^2 =0$
 in terms of $\phi$ and $\theta$.
Evaluating $\|\tilde{f}_{\phi,\theta}(t)\|^2$ at $\phi=\phi(t)$ and $\theta=\theta(t)$
 yields the envelopes for the norm of solutions $\tilde{f}(t)$ of ODE~\eqref{ODE:reduced}.
Consequently, we derive the envelopes $h_\pm(t) \|f^I\|^2$ for the original problem,
 since $\|f(t)\|_2 = e^{-\Re \lambda_1^\CC t} \|\tilde{f}(t)\|_2$.
\qed}
\end{proof}
\tw{
\begin{corollary}
Let $\CC\in\C^{2\times 2}$ be a diagonalizable, positive stable matrix.
Then the minimal multiplicative constant~$c$ in~\eqref{exp-decay:c} for the ODE~\eqref{ODE:general:n}
 is given by
 \begin{equation} \label{min:c:implicit} 
   c = \sqrt{ \sup_{t\geq 0} m_+(t)}\,,
 \end{equation} 
 where $m_+(t)$ is the function given in Proposition~\ref{bestConst-2D:3}.
\end{corollary}
In general we could not find an explicit formula for $\sup_{t\geq 0} m_+(t)$.
}

%%%%%%%%%%%%%%%%%%%%%%%%%%%%%%%%%%%%%%%%%%%%%%%%%%%%%%%%%%%%%%%%%%%%%%%%%%%%%%%%%%%%%%%%%%%%%%%%

\section{A Family of Decay Estimates for Hypocoercive ODEs}\label{sec:ODE-ex}

In this section we shall illustrate the interdependence of maximizing the decay rate $\lambda$ and minimizing the multiplicative constant $c$ in estimates like \eqref{exp-decay:c}. 
For the ODE-system \eqref{ODE:general:n},
 the procedure described in Remark \ref{rem1.2}(b) yields the optimal bound for large time,
 with the sharp decay rate  $\mu:=\min\{\Re(\lambda)|\lambda$ is an eigenvalue of $\CC\}$. 
But for non-coercive $\CC$ we must have $c>1$. 
Hence, such a bound cannot be sharp for short time. 
As a counterexample we consider the simple energy estimate
 (obtained by premultiplying \eqref{ODE:general:n} with $f^*$)
$$
  \|f(t)\|_2 \le e^{-\mu_s t}\|f^I\|_2\,,\quad t\ge0\,,
$$
with $\CC_s:=\frac12(\CC+\CC^*)$ and $\mu_s:=\min\{\lambda|\lambda$ is an eigenvalue of $\CC_s\}$. 

The goal of this section is to derive decay estimates for~\eqref{ODE:general:n} with rates in between this weakest rate $\mu_s$ and the optimal rate $\mu$ from \eqref{ODE:exponentialDecay}. \tw{It holds that $\mu_s\le\mu$.}
At the same time we shall also present \emph{lower bounds} on $\|f(t)\|_2$. 
The energy method again provides the simplest example of it, in the form
$$
  \|f(t)\|_2 \ge e^{-\nu_s t}\|f^I\|_2\,,\quad t\ge0\,,
$$
with $\nu_s:=\max\{\lambda|\lambda$ is an eigenvalue of $\CC_s\}$. Clearly, estimates with decay rates outside of $[\mu_s,\nu_s]$ are irrelevant. 

We present our main result only for the two-dimensional case, as the best multiplicative constant is not yet known explicitly in higher dimensions (cf. \S\ref{sec:min}):

\begin{proposition}\label{prop:5.1}
  Let $\CC\in\C^{2\times 2}$ be a diagonalizable positive stable matrix with spectral gap $\mu:=\min\{\Re(\lambda_j^\CC)|\,j=1,2\}$. 
  Then, all solutions to \eqref{ODE:general:n} satisfy the following upper and lower bounds:
  \begin{itemize}
   \item[a)] 
   \begin{equation}\label{up-bound}
     \|f(t)\|_2 \le c_1(\tilde\mu)\, e^{-\tilde\mu t}\|f^I\|_2\,,\quad t\ge0\,,\quad \mu_s\le\tilde\mu\le\mu\,,
   \end{equation}
%    with 
%    $$
%      c_1^2(\tilde\mu)=\kappa\big(\PP(\tilde\mu)\big) 
%    $$
%    given explicitly in \eqref{kappa-min} below, and $\PP(\tmu)=\WW\BB_u\WW^*$, 
%    $$
%      \WW = \begin{pmatrix} 1 & \alpha \\ 0 & \sqrt{1-\alpha^2} \end{pmatrix} \,,\quad
%      \BB_u = \begin{pmatrix} 1 & \beta(\tmu) \\ \beta(\tmu) & 1 \end{pmatrix} \,.
%    $$
   with 
   $$
     c_1^2(\tilde\mu)=\kappa_{\min}(\beta(\tmu))
   $$
   given explicitly in \eqref{kappa-min} below.
   There, $\alpha\in[0,1)$ is the $\cos$ of the (minimal) angle of the eigenvectors of $\CC^*$ (cf. the proof of Lemma \ref{bestP-2D}),
   and $\beta(\tmu)= \max(-\alpha,-\beta_0)$, with $\beta_0$ defined in \eqref{cond2}, \eqref{cond1} below.
   \item[b)] 
   \begin{equation}\label{low-bound}
     \|f(t)\|_2 \ge c_2(\tilde\mu)\, e^{-\tilde\mu t}\|f^I\|_2\,,\quad t\ge0\,,\quad \nu\le\tilde\mu\le\nu_s\,,
   \end{equation}
   with $\nu:=\max\{\Re(\lambda_j^\CC)|\,j=1,2\}$. The maximal constant 
   $$
     c_2^2(\tilde\mu)=\kappa_{\min}(\beta(\tmu))^{-1}
   $$
   is given again by \eqref{kappa-min}, with $\alpha$, $\beta(\tmu)$ defined as in Part (a).
  \end{itemize}
\end{proposition}
\begin{proof}
\underline{Part (a):} For a fixed $\tmu\in[\mu_s,\mu]$ we have to determine the smallest constant $c_1$ for the estimate \eqref{up-bound},
 following the strategy of proof from \S\ref{sec:min}. 
To this end,
 we use a unitary transformation of the coordinate system and write $\PP(\tmu)=\WW\BB_u\WW^*$ with 
 \begin{equation} \label{matrices:W:Bu}
  \WW = \begin{pmatrix} 1 & \alpha \\ 0 & \sqrt{1-\alpha^2} \end{pmatrix} \,,\quad
  \BB_u = \begin{pmatrix} 1/b & \beta(\tmu) \\ \bar\beta(\tmu) & b \end{pmatrix} \,,
 \end{equation}
where we set w.l.o.g.\ $b_1=1/b$, $b_2=b$ with $b>0$. 
Moreover, $|\beta|^2<1$ has to hold. 
Now, we have to find the positive definite Hermitian matrix $\BB_u$,
 such that the analog of \eqref{matrixestimate2}, \eqref{matrixestimate2b} holds, i.e.:
\begin{equation}\label{matrixA}
  \AA:=\begin{pmatrix}
         2(\Re(\lambda_1^\CC)-\tmu)/b & (\bar\lambda_1^\CC+\lambda_2^\CC-2\tmu) \beta \\
         (\lambda_1^\CC+\bar\lambda_2^\CC-2\tmu) \bar\beta & 2(\Re(\lambda_2^\CC)-\tmu)b
        \end{pmatrix} \ge0\,,
\end{equation}
As in the proof of Lemma \ref{bestP-2D-diagonalB},
 we assume that the eigenvalues of $\CC$ are ordered as $\Re(\lambda_2^\CC)\ge \Re(\lambda_1^\CC)=\mu\ge\tmu$. 
Hence, $\trace\AA\ge0$. For the non-negativity of the determinant to hold, i.e.
\begin{equation}\label{detA}
  \det\AA=4\big(\Re(\lambda_1^\CC)-\tmu\big)\big(\Re(\lambda_2^\CC)-\tmu\big)-|\lambda_1^\CC+\bar\lambda_2^\CC-2\tmu|^2|\beta|^2 \ge0\,,
\end{equation}
we have the following restriction on $\beta$:
\begin{equation}\label{cond2}
  |\beta|^2\le\beta_0^2 := \frac{4\big(\Re(\lambda_1^\CC)-\tmu\big)\big(\Re(\lambda_2^\CC)-\tmu\big)}{|\lambda_1^\CC+\bar\lambda_2^\CC-2\tmu|^2}\,.
\end{equation}
If $\lambda_1^\CC+\bar\lambda_2^\CC-2\tmu=0$,
 we conclude $\lambda_1^\CC=\lambda_2^\CC$ and that we have chosen the sharp decay rate $\tmu=\mu$. 
As the associated, minimal condition number $\kappa(\PP)$ was already determined in Lemma \ref{bestP-2D},
 we shall not rediscuss this case here. 
But to include this case into the statement of the theorem, we set 
\begin{equation}\label{cond1}
  \beta_0:=1\,,\quad \mbox{if }\lambda_1^\CC=\lambda_2^\CC \mbox{ and }\tmu=\mu\,.
\end{equation}

{}From \eqref{cond2} we conclude that $\beta_0\in[0,1]$. 
Note that $\beta_0=1$ is only possible for $\tmu=\mu$ and $\lambda_1^\CC=\lambda_2^\CC$, i.e.\ the case that we just sorted out.
For the rest of the proof we hence assume that condition \eqref{cond2} holds with $\beta_0\in[0,1)$. 

For admissible matrices $\BB_u$ (i.e.\ with $b>0$ and $|\beta|\le\beta_0$) it remains to determine the matrix 
\begin{equation*} %\label{def:P-b}
  \PP(b,\beta)=\WW\BB_u\WW^*
  = \begin{pmatrix}
            \tfrac1b + 2\alpha\Re\beta+ b\alpha^2 & \;\;(\beta+b\alpha) \sqrt{1-\alpha^2} \\
            (\bar\beta+b\alpha) \sqrt{1-\alpha^2} & b (1-\alpha^2)
           \end{pmatrix}\,,
\end{equation*}
(with $\WW$ and $\BB_u$ given in \eqref{matrices:W:Bu}), having the minimal condition number
$\kappa\big(\PP(b,\beta)\big) =\lambda^\PP_+(b,\beta) /\lambda^\PP_-(b,\beta)$. 
Here
 \[ \lambda_\pm^\PP(b,\beta) = \frac{\trace\PP(b,\beta)\pm \sqrt{(\trace\PP(b,\beta))^2 -4\det \PP(b,\beta)}}{2} \]
are the (positive) eigenvalues of $\PP(b,\beta)$.  

As a first step we shall minimize $\kappa\big(\PP(b,\beta)\big)$ w.r.t.\ $b$ (and for $\beta$ fixed),
 since $\argmin_{b>0} \kappa\big(\PP(b,\beta)\big)$ will turn out to be independent of $\beta$.
We notice that $\trace\PP(b,\beta) = b+ 2\alpha\Re\beta+1/b$ is a convex function of $b\in(0,\infty)$ 
which attains its minimum for $b=1$.
Moreover, $\det\PP(b,\beta)=(1-\alpha^2)(1-|\beta|^2)>0$ is independent of $b$.
This yields the condition number 
 \[
  \kappa_{\min} (\beta) =\frac{\lambda^\PP_+(1,\beta)}{\lambda^\PP_-(1,\beta)}
              =\frac{1 + \sqrt{1 -\frac{(1-\alpha^2)(1-|\beta|^2)}{(1+\alpha\Re\beta)^2}}}{1 - \sqrt{1 -\frac{(1-\alpha^2)(1-|\beta|^2)}{(1+\alpha\Re\beta)^2}}}\,.
 \]

As a second step we minimize $\kappa_{\min} (\beta)$ on the disk $|\beta|\le\beta_0$. To this end, the quotient $\frac{(1-\alpha^2)(1-|\beta|^2)}{(1+\alpha\Re\beta)^2}$ should be as large as possible. For any fixed $|\beta|\le\beta_0$, this happens by choosing $\beta=-|\beta|$, since $\alpha\in[0,1)$. Hence it remains to maximize the function $g(\beta):=\frac{1-\beta^2}{(1+\alpha\beta)^2}$ on the interval $[-\beta_0,0]$. It is elementary to verify that $g$ is maximal at $\tilde\beta:=\max(-\alpha,-\beta_0)$. Then, the minimal condition number is
\begin{equation} \label{kappa-min}
  \kappa_{\min} (\tilde\beta) =\kappa\big(\PP(1,\tilde\beta)\big)=
  \frac{1 + \sqrt{1 -\frac{(1-\alpha^2)(1-\tilde\beta^2)}{(1+\alpha\tilde\beta)^2}}}{1 - \sqrt{1 -\frac{(1-\alpha^2)(1-\tilde\beta^2)}{(1+\alpha\tilde\beta)^2}}}\,.
\end{equation}

\noindent
\underline{Part (b):} Since the proof of the lower bound is very similar to Part (a), we shall just sketch it.
For a fixed $\tmu\in[\nu,\nu_s]$ we have to determine the largest constant $c_2$ for the estimate \eqref{low-bound}. 
To this end we need to satisfy the inequality
$$
  \CC^*\PP+\PP\CC \le 2\tmu\PP
$$
with a positive definite Hermitian matrix $\PP$ with minimal condition number $\kappa(\PP)$. 
In analogy to \S\ref{sec:Lyap} this would imply 
$$
  \ddt \|f(t)\|^2_\PP  \geq -2\tmu \|f(t)\|^2_\PP\,, 
$$
and hence the desired lower bound
\begin{equation*} %\label{ODE:exponentialDecay-low}
 \|f(t)\|^2_2\geq (\lambda^\PP_{\max})^{-1} \|f(t)\|^2_\PP \geq  (\lambda^\PP_{\max})^{-1} e^{-2\tmu t} \|f^I \|^2_\PP \geq  (\kappa(\PP))^{-1}\, e^{-2\tmu t} \|f^I \|^2_2 \,. %,\quad t\ge0\,.
\end{equation*}  

For minimizing $\kappa(\PP)$, we again use a unitary transformation of the coordinate system and write $\PP$ as
$\PP(\tmu)=\WW\BB_l\WW^*$, with $\WW$ from \eqref{matrices:W:Bu} and the positive definite Hermitian matrix
   $$
     \BB_l = \begin{pmatrix} 1/b & \beta(\tmu) \\ \bar\beta(\tmu) & b \end{pmatrix} \,,
   $$
with $b>0$ and $|\beta|^2<1$.
Then, the matrix $\AA$ from \eqref{matrixA} has to satisfy $\AA\le0$. Since we chose 
the eigenvalues of $\CC$ to be ordered as $\Re(\lambda_1^\CC)\le \Re(\lambda_2^\CC)=\nu\le\tmu$, we have $\trace\AA\le0$. 
The necessary non-negativity of its determinant again reads as \eqref{detA}. 

In the special case $\lambda_1^\CC+\bar\lambda_2^\CC-2\tmu=0$, we conclude again $\lambda_1^\CC=\lambda_2^\CC$ and $\tmu=\nu$. 
Hence $\AA=0$. Since $\beta$ is then only restricted by $|\beta|<1$, we can again set $\beta_0=1$ and obtain the minimal $\kappa(\PP)$ for $\tilde\beta(\nu)=-\alpha$, as in Part (a).

In the generic case, the minimal $\kappa(\PP)$ is obtained for $\tilde\beta=\max(-\alpha,-\beta_0)$ with $\beta_0$ given in \eqref{cond2}. 
Hence, the maximal constant in the lower bound \eqref{low-bound} is $c_2^2(\tilde\mu)=\kappa_{\min}(\tilde\beta)^{-1}$ where $\kappa_{\min}$ is given by \eqref{kappa-min}. 
This finishes the proof.
\qed
\end{proof}

%%%%%%%%%%%%%%%%%%%%%%%%%%%%%%%%%%%%%%%%%%%%%%5

We illustrate the results of Proposition~\ref{prop:5.1} with two examples.
\begin{example} \label{ex:ODE1}
We consider ODE~\eqref{ODE:general:n} with the matrix
\begin{equation*} %\label{matrixC:ex1}
 \CC = \begin{pmatrix} 1 & -1 \\ 1 & 0 \end{pmatrix}
\end{equation*}
which has eigenvalues $\lambda_\pm = (1\pm i\sqrt{3})/2$, and some normalized eigenvectors of~$\CC^*$ are, e.g.
\begin{equation}\label{ex1:eigenvectors}
 w_+ = \frac1{\sqrt2} \begin{pmatrix} -1 \\ \lambda_- \end{pmatrix}\,, \qquad
 w_- = \frac1{\sqrt2} \begin{pmatrix} -\lambda_- \\ 1 \end{pmatrix}\,.
\end{equation}
The optimal decay rate is $\mu=1/2$,
 whereas the minimal and maximal eigenvalues of $\CC_s$ are $\mu_s=0$ and $\nu_s=1$, respectively.
To bring the eigenvectors of $\CC^*$ in the canonical form used in the proof of Proposition~\ref{prop:5.1},
 we fix the eigenvector~$w_+$,
 and choose the unitary multiplicative factor for the second eigenvector~$w_-$ as in~\eqref{ex1:eigenvectors}
 such that $\langle w_+, w_- \rangle$ is a real number. 
Finally, we use the Gram-Schmidt process to obtain a new orthonormal basis
 such that the eigenvectors of $\CC^*$ in the new orthonormal basis are of the form~\eqref{eigenvectors:newBasis} with $\alpha=1/2$.
Then, the upper and lower bounds for the Euclidean norm of a solution of~\eqref{ODE:general:n} are plotted in Fig.~\ref{fig:estimate-of-norm} and Fig.~\ref{fig:estimate-of-norm+zoom}.
\begin{figure}[h!]
\includegraphics[width=\textwidth]{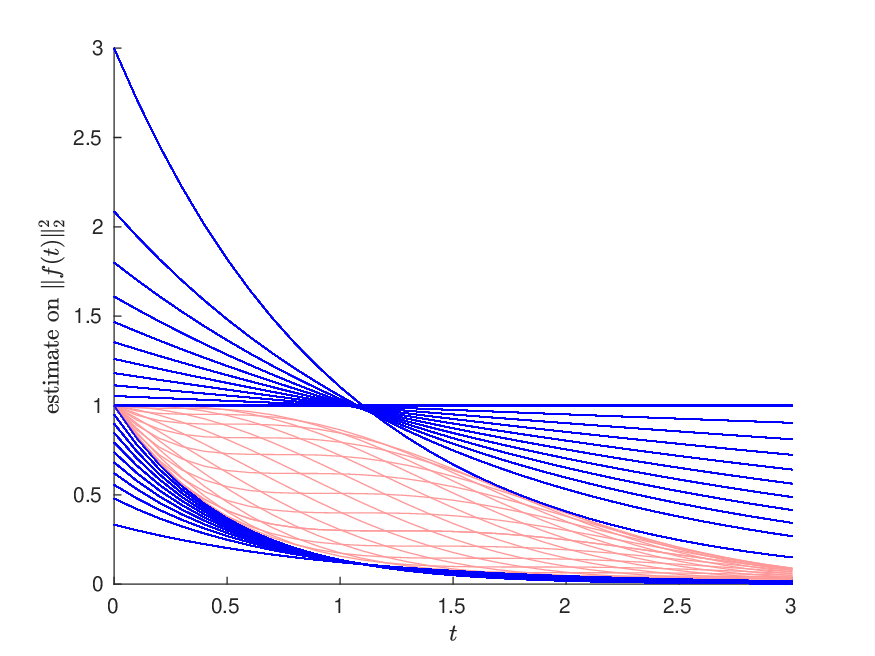}
\caption{
 The red (grey) curves are the squared norm of solutions $f(t)$ for ODE~\eqref{ODE:general:n}
 with matrix $\CC=[1,-1; 1,0]$ and various initial data~$f^I$ with norm~$1$.
 The blue (black) curves are the lower and upper bounds for the squared norm of solutions.
 Note: The curves are colored only in the electronic version of this article.
}
\label{fig:estimate-of-norm}
\end{figure}
\begin{figure}[h!]
\includegraphics[width=\textwidth]{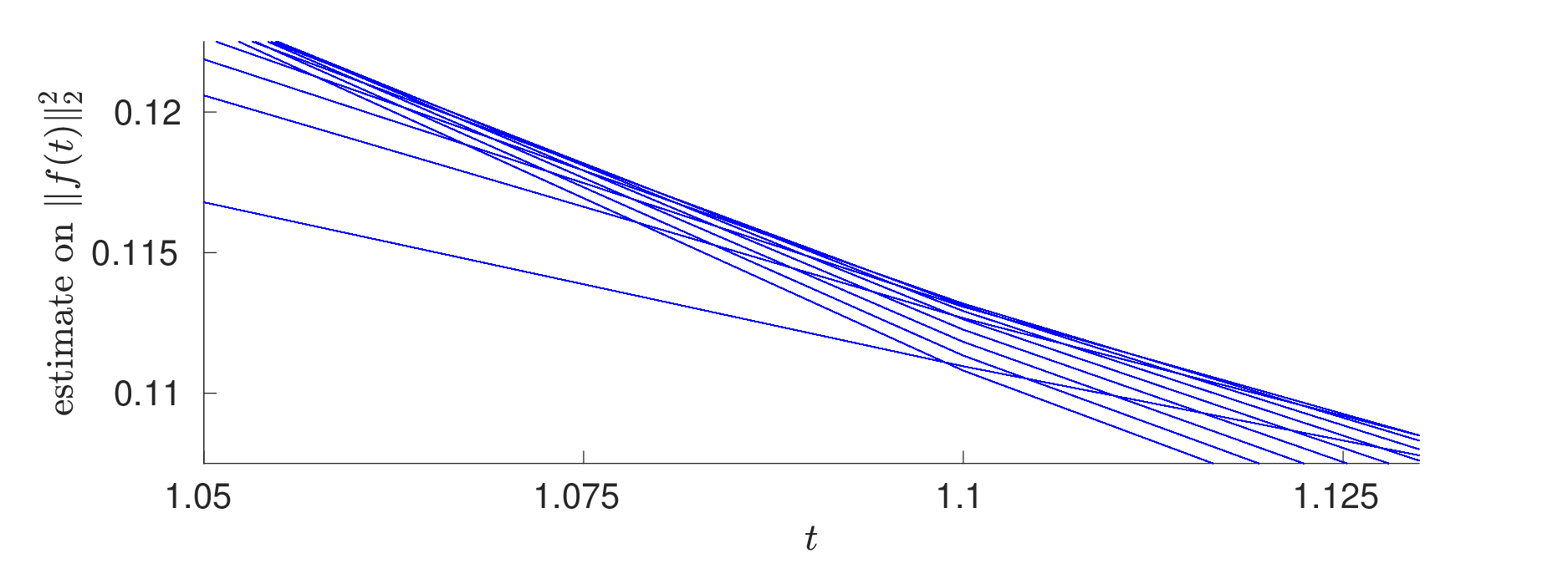}
\caption{
Zoom of Fig.~\ref{fig:estimate-of-norm}: The curves are the lower bounds for the squared norm of solutions for ODE~\eqref{ODE:general:n}
 with matrix $\CC=[1,-1; 1,0]$ and various initial data~$f^I$ with norm~$1$.
This plot shows that these lower bounds do not intersect in a single point.}
\label{fig:estimate-of-norm+zoom}
\end{figure}
For both the upper and lower bounds,
 the respective family of decay curves does \emph{not} intersect in a single point (see Fig.~\ref{fig:estimate-of-norm+zoom}).
Hence, the whole family of estimates provides a (slightly) better estimate on $\|f(t)\|_2$
 than if just considering the two extremal decay rates.
\tw{For the upper bound this means
 \[
  \|f(t)\|_2
   \leq \min_{\tilde\mu\in[\mu_s,\mu]} c_1(\tilde\mu)\, e^{-\tilde\mu t} \, \|f^I\|_2 
%   \leq \min\{c_1(\mu_s)\, e^{-\mu_s t}\,, c_1(\mu)\, e^{-\mu t}\} \, \|f^I\|_2 \,,\quad t\ge0\,,
   \leq \min\{1\,, c_1(\mu)\, e^{-\mu t}\} \, \|f^I\|_2 \,,\quad t\ge0\,,
 \]
 and for the lower bound
 \[
  \|f(t)\|_2
   \geq \max_{\tilde\nu\in[\nu,\nu_s]} c_2(\tilde\nu)\, e^{-\tilde\nu t} \, \|f^I\|_2 
   \geq \max\{c_2(\nu)\, e^{-\nu t}\,, c_2(\nu_s)\, e^{-\nu_s t}\} \, \|f^I\|_2 \,,\ \ t\ge0\,.
 \]
%  i.e. $\mu$ and $\mu_s$ for the upper bound 
%  and $\nu$ and $\nu_s$ for the lower bound.
}

Note that the upper bound $\sqrt3 e^{-t/2}$ with the sharp decay rate $\mu=\frac12$ carries the optimal multiplicative constant $c=\sqrt3$,
 as it touches the set of solutions (see Fig.~\ref{fig:estimate-of-norm}). 
But this is not true for the estimates with smaller decay rates (except of $\tmu=0$). 
\tw{The reason for this lack of sharpness is the fact that the inequality
 $\|f(t)\|^2_\PP \leq e^{-2\tilde\mu t} \|f^I \|^2_\PP$
 used in the proof of Proposition~\ref{prop:5.1} is, in general, 
 not an equality (in contrast to~\eqref{est:middle}). \qed}
 \end{example}

%\tw{
In the next example we consider a matrix $\CC\in\R^{2\times2}$ with $\Re\lambda_1\ne\Re\lambda_2$,
 which corresponds to the case analyzed in Theorem~\ref{thm:bestConst-2D:2}. 
For such cases the strategy of Proposition~\ref{prop:5.1} (based on minimizing $\kappa(\PP)$) could be improved in the spirit of Theorem~\ref{thm:bestConst-2D:2},
 but we shall not carry this out here. Hence, the estimates of the following example will not be sharp, see Fig.~\ref{fig:ex2:estimate-of-norm}.
%}%end\tw

% \begin{example}
% \begin{figure}
% \includegraphics[width=\textwidth]{estimate-of-norm-matrix2}
% \caption{
%  The red (grey) curves are the squared norm of solutions $f(t)$ for the ODE~\eqref{ODE:general:n}
%  with matrix $\CC=[3/4,-1; 1,-1/4]$ and various initial data~$f^I$ with norm $1$.
%  The blue (black) curves are the lower and upper bounds for the squared norm of solutions.
%  Note: The curves are colored only in the electronic version of this article.
% }
% \end{figure}
% \end{example}

\begin{example} \label{ex:ODE2}
We consider ODE~\eqref{ODE:general:n} with the matrix
\begin{equation*} %\label{matrixC:ex2}
 \CC = \begin{pmatrix} 19/20 & -3/10 \\ 3/10 & -1/20 \end{pmatrix}
% \CC = \begin{pmatrix} 1 & -3/10 \\ 3/10 & 0 \end{pmatrix}
\end{equation*}
which has the eigenvalues $\lambda_1=1/20$ and $\lambda_2=17/20$, 
%which has the eigenvalues $\lambda_1=1/10$ and $\lambda_2=9/10$, 
and some normalized eigenvectors of $\CC^*$ are, e.g.
\begin{equation*}%\label{ex2:eigenvectors}
 w_1 = \frac1{\sqrt{10}} \begin{pmatrix} 1 \\ -3 \end{pmatrix}\,, \qquad
 w_2 = \frac1{\sqrt{10}} \begin{pmatrix} 3 \\ -1 \end{pmatrix}\,.
\end{equation*}
The optimal decay rate is $\mu=1/20$, %$\mu=1/10$,
whereas the minimal and maximal eigenvalues of $\CC_s$ are $\mu_s=-1/20$ and $\nu_s=19/20$, 
% $\CC_s$ are $\mu_s=0$ and $\nu_s=1$, 
 respectively.
Since the matrix $\CC$ and its eigenvalues are real valued,
 the eigenvectors of $\CC^*$ are already in the canonical form used in the Gram-Schmidt process to obtain a new orthogonal basis
 such that the eigenvectors of $\CC^*$ in the new basis are of the form~\eqref{eigenvectors:newBasis} with $\alpha=3/5$.
Then, the upper and lower bounds for the Euclidean norm of a solution of~\eqref{ODE:general:n} are plotted in Fig.~\ref{fig:ex2:estimate-of-norm}. 
Since $\mu_s<0$, solutions~$f(t)$ to this example may initially increase in norm. \qed
\begin{figure}[h!]
\includegraphics[width=\textwidth]{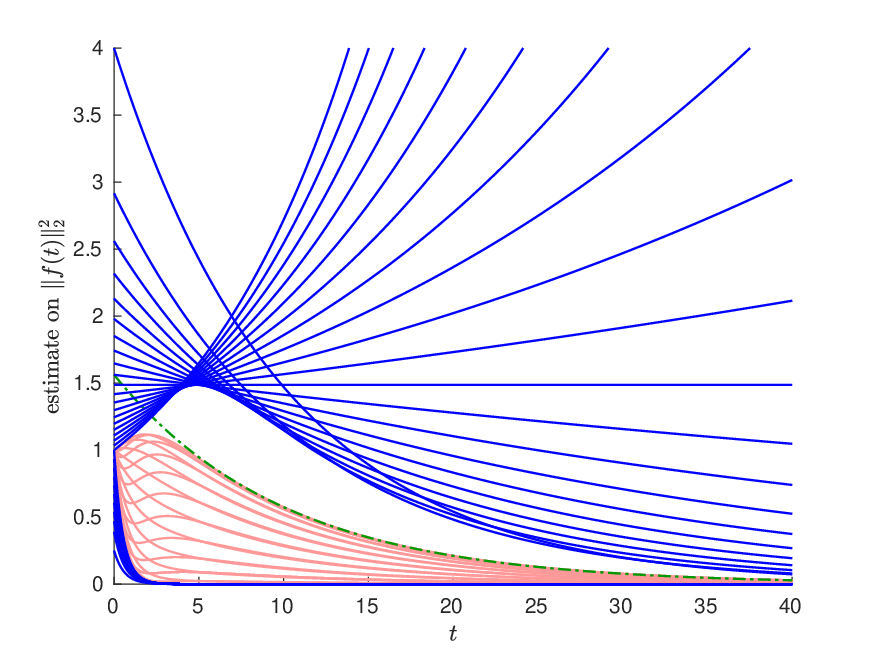}
\caption{
 The red (grey) curves are the squared norm of solutions $f(t)$ for ODE~\eqref{ODE:general:n}
 with matrix $\CC=[19/20,-3/10; 3/10,-1/20]$ % $\CC=[1,-3/10; 3/10,0]$ 
 and various initial data~$f^I$ with norm~$1$.
 The blue (black) curves are the lower and upper bounds for the squared norm of solutions
  derived from Proposition~\ref{prop:5.1}.
 The green (black) \emph{dash-dotted curve} is the upper bound for the squared norm of solutions 
  derived from Theorem~\ref{thm:bestConst-2D:2}.
 Note: The curves are colored only in the electronic version of this article.
}
\label{fig:ex2:estimate-of-norm}
\end{figure}
\end{example}

%%%%%%%%%%%%%%%%%%%%%%%%%%%%%%%%%%%%%%%%%%%%%%%%%%%%%%%%%%%%%%%%%%%%%%%%%%%%%%%%%%%%%%%%%%%%%%%%

\noindent
\textbf{Acknowledgments.}
All authors were supported by the FWF-funded SFB \#F65. 
The second author was partially supported by the FWF-doctoral school W1245 ``Dissipation and dispersion in nonlinear partial differential equations''.
We are grateful to the anonymous referee who led us to better distinguish the different cases studied in \S3 and \S4.

%
% ---- Bibliography ----
%
%Compiling sequence to make citations from .bib file work: PDFLaTeX -> BibTeX -> PDFLaTeX -> PDFLaTeX (not quickbuild or latex!)
%other styles: acm, abbrv, simple

\bibliographystyle{spmpsci} %other styles: acm, abbrv, simple

\begin{thebibliography}{99}

\providecommand{\url}[1]{{#1}}
\providecommand{\urlprefix}{URL }
\expandafter\ifx\csname urlstyle\endcsname\relax
  \providecommand{\doi}[1]{DOI~\discretionary{}{}{}#1}\else
  \providecommand{\doi}{DOI~\discretionary{}{}{}\begingroup
  \urlstyle{rm}\Url}\fi

\bibitem{AAC16}
Achleitner, F., Arnold, A., Carlen, E.A.: 
On linear hypocoercive {BGK} models.
\newblock Gon\c{c}alves P., Soares A. (eds) From Particle Systems to Partial Differential Equations {III}.   \newblock Springer Proc. Math. Stat., vol 162, pp. 1--37. 
\newblock Springer, Cham (2016)

\bibitem{AAC17}
Achleitner, F., Arnold, A., Carlen, E.A.: 
On multi-dimensional hypocoercive {BGK} models.
\newblock Kinet. Relat. Models \textbf{11}, 953--1009 (2018) %Volume no.4
%\newblock {arXiv preprint}, arXiv:1711.07360\/ (2017).

\bibitem{AAS} 
Achleitner, F., Arnold,  A., St{\"{u}}rzer, D.: 
Large-Time Behavior in Non-Symmetric Fokker-Planck Equations.
%Rivista di Matematica della Universit\`a di Parma 
\newblock Riv. Math. Univ. Parma (N.S.) \textbf{6}, 1--68 (2015)

\bibitem{ArEr14}
{Arnold, A., Erb, J.:}
\newblock Sharp entropy decay for hypocoercive and non-symmetric {F}okker-{P}lanck equations with linear drift.
\newblock {arXiv preprint}, arXiv:1409.5425\/ (2014)

%\bibitem{AEW17}
%Arnold, A., Einav, A., W{\"{o}}hrer, T.: {On the rates of decay to equilibrium
%  in degenerate and defective Fokker-Planck equations}.
%\newblock Preprint,  arXiv:1709.10216 (2017).

\bibitem{AJW18}
{Arnold, A., Jin, S., W\"ohrer, T.:}
\newblock Sharp Decay Estimates in Defective Evolution Equations: from {ODE}s to Kinetic {BGK} Equations.
\newblock {preprint}, (2018)

%\bibitem{A73}
%{V. I. Arno\v{l}d},
%\newblock ``Ordinary differential equations'',
%\newblock {MIT Press, Cambridge, Mass.-London}, (1973).

\bibitem{Arnold1978}
Arnold, V.I.: Ordinary differential equations.
\newblock MIT Press, Cambridge, Mass.-London (1978)
%\newblock Translated from the Russian and edited by Richard A. Silverman

\bibitem{BGK}
Bhatnagar, P.L., Gross, E.P., Krook, M.: 
A Model for Collision Processes in Gases. 
I. Small Amplitude Processes in Charged and Neutral One-Component Systems.
\newblock Phys. Rev. \textbf{94}, 511--525 (1954) %Volume no. 3

\bibitem{BoVa04}
Boyd, S.P., Vandenberghe, L.: 
Convex Optimization. 
Cambridge University Press, Cambridge (2004)

\bibitem{BrMo94}
Braatz, R.D., Morari, M.: 
Minimizing the {E}uclidean condition number.
\newblock {SIAM J. Control Optim.} \textbf{32}, 1763--1768 (1994) %Volume no. 6

\bibitem{Bu68}
Businger, P.A.: 
Matrices which can be optimally scaled.
\newblock {Numer. Math.} \textbf{12}, 346--348 (1968)

\bibitem{DMS15}
Dolbeault, J., Mouhot, C., Schmeiser, C.: 
Hypocoercivity for linear kinetic equations conserving mass.
%\newblock Transactions of the American Mathematical Society
\newblock Trans. Amer. Math. Soc. \textbf{367}, 3807--3828 (2015) %Volume no. 6
%\newblock \doi{10.1090/S0002-9947-2015-06012-7}.
%\newblock
%  \urlprefix\url{http://arxiv.org/abs/1005.1495{\%}5Cnhttp://linkinghub.elsevier.com/retrieve/pii/S1631073X09000880{\%}5Cnhttp://www.ams.org/tran/2015-367-06/S0002-9947-2015-06012-7/}

%\bibitem{Hes09}
%{J. P. Hespanha},
%\newblock ``Linear systems theory'',
%\newblock {Princeton University Press, Princeton}, %NJ,
%(2009).

%\bibitem{Jin2017}
%Jin, S., Zhu, Y.: {Hypocoercivity and Uniform Regularity for the
%  Vlasov-Poisson-Fokker-Planck System with Uncertainty and Multiple Scales}.
%\newblock Preprint  (2017)

\bibitem{Ko2006}
Kolotilina, L.Yu.:
\newblock Solution of the problem of optimal diagonal scaling for quasireal Hermitian positive definite $3\times 3$ matrices.
\newblock J. Math. Sci. (N.Y.) \textbf{132}, 190--213 (2006) %Volume no. 2
%\newblock Journal of Mathematical Sciences \textbf{132}, no. 2, 190--213 (2006).

%\bibitem{Monmarche2015}
%Monmarch{\'{e}}, P.: {Generalized $\Gamma$ calculus and application to
%  interacting particles on a graph}.
%\newblock Preprint,	arXiv:1510.05936 (2015).

\tw{\vspace{-4mm}
\bibitem{MiMo2013}
Miclo, L., Monmarch{\'e}, P.:
\newblock {{\'E}tude spectrale minutieuse de processus moins ind{\'e}cis que les autres},
\newblock In: Donati-Martin C., Lejay A., Rouault A. (eds) {S{\'e}minaire de Probabilit{\'e}s XLV},
\newblock Lecture Notes in Mathematics, vol 2078, pp. 459--481.
\newblock Springer, Heidelberg (2013).
\newblock English summary available at \url{https://www.ljll.math.upmc.fr/~monmarche}
}

\bibitem{SeOv1990}
Sezginer, R.S., Overton, M.L.:
{The largest singular value of $e^X A_0 e^{-X}$ is convex on convex sets of commuting matrices},
\newblock IEEE Trans. Automat. Control \textbf{35}, 229--230 (1990) %Volume no. {2}
%\newblock IEEE Transactions on Automatic Control \textbf{35}, no. {2}, 229--230 (1990).

\bibitem{ViH06}
Villani, C.:
\newblock Hypocoercivity.
\newblock {Mem. Amer. Math. Soc.},
\textbf{202} (2009) %, %no. 950
%iv+141 pp.

\end{thebibliography}

 %library=name of .bib file

\end{document}